\documentclass[11pt,leqno]{amsart}

\usepackage{amsmath,amsthm}
\usepackage {latexsym}
\usepackage{amssymb}

\newcommand{\supp}{\mbox{\rm supp}}

\newcommand{\bear}{\begin{eqnarray}}
\newcommand{\eear}{\end{eqnarray}}
\newcommand{\eps}{{\varepsilon}}
\newcommand{\coker}{{\rm coker}\,}
\newcommand{\R}{{\mathbb R}}
\newcommand{\T}{{\mathbb T}}
\newcommand{\Z}{{\mathbb Z}}

\newcommand{\Compl}{{\mathbb C}}

\newcommand{\les}{\lesssim}
\newcommand{\gtr}{\gtrsim}

\newtheorem{theorem}{Theorem}
\newtheorem{lemma}[theorem]{Lemma}

\newtheorem{corollary}[theorem]{Corollary}

\newtheorem{proposition}[theorem]{Proposition}

\theoremstyle{remark}
\newtheorem{remark}{Remark}
\def\Uplus{U^+}
\def\Uminus{U^-}

\def\Rminus{R^-}

\def\la{\langle}
\def\ra{\rangle}
\def\norm[#1][#2]{\|#1\|_{#2}}
\def\bignorm[#1][#2]{\Big\|#1\Big\|_{#2}}
\def\japanese[#1]{\langle #1 \rangle}
\def\Im[#1]{{\rm Im}(#1)}
\def\Re[#1]{{\rm Re}(#1)}
\newcommand{\txt}{\textstyle}

\begin{document}

\title[The Schr\"odinger equation with time-periodic $L^{n/2}$ potentials]
{Strichartz estimates for the Schr\"odinger equation with time-periodic
$L^{n/2}$ potentials}
\date{August 10, 2007}

\author{Michael\ Goldberg}
\thanks{The author received partial support from NSF grant DMS-0600925 during
the preparation of this work.}
\address{Department of Mathematics, Johns Hopkins University,
3400 N. Charles St., Baltimore, MD 21218}
\email{mikeg@math.jhu.edu} 

\begin{abstract}
We prove Strichartz estimates for the Schr\"odinger operator 
$H = -\Delta + V(t,x)$ with time-periodic complex potentials $V$ belonging to
the scaling-critical space $L^{n/2}_x L^\infty_t$ in dimensions
$n \ge 3$.  This is done directly from estimates on the resolvent rather
than using dispersive bounds, as the latter generally require a stronger 
regularity condition than what is stated above.  In typical fashion, we project
onto the continuous spectrum of the operator and must assume an absence of
resonances.  Eigenvalues are permissible at any location in the spectrum, 
including at threshold energies, provided that the associated eigenfunction
decays sufficiently rapidly.  
\end{abstract}

\maketitle

\section{Introduction}
The past decade has seen considerable progress in identifying classes of
Schr\"odinger operators which retain the same dispersive properties as the
Laplacian.  In many cases these operators are described by a simple 
perturbation of the Laplacian, taking the form $H = -\Delta + L(t,x)$.
Typically $L$ is a linear self-adjoint differential operator of degree 
$d = 0,1,2$ representing electrostatic, magnetic, and/or geometric
perturbations, respectively.  In this paper we consider the Floquet-type
potential $L(t,x) = V(t,x)$ satisfying $V(t+2\pi,x) = V(t,x)$ for all 
$t \in \R$ and $x \in \R^n$.

We do not assume any self-adjointness in our main theorem, instead allowing
$V$ to be a complex-valued function.  Further improvements for real and/or
time-independent potentials are examined as corollaries and applications
of the first result.

The propagator $e^{-it\Delta}$ of the free Schr\"odinger equation in $\R^n$
may be represented as a convolution operator with kernel $(4\pi it)^{-n/2}
e^{-i(|x|^2/4t)}$.  From this formula it is clear that the free evolution
satisfies the dispersive bound $\norm[e^{it\Delta}][1\to\infty] \le 
(4\pi|t|)^{-n/2}$ at all times $t \not= 0$.  A $TT^*$ argument combined with
fractional integration bounds for the $t$ variable then leads to
the family of Strichartz inequalities
\begin{equation} \label{eq:Strichartz}
\norm[e^{-it\Delta}u_0][L^p_tL^q_x] \le C_p\norm[u_0][2], \quad
\frac{2}{p} + \frac{n}{q} = \frac{n}{2}, \quad p,q\in [2,\infty]
\end{equation}
for all $u_0 \in L^2(\R^n)$. 
To be precise, the $p=2$ endpoint
requires a more detailed computation~\cite{KT} and is false when $n=2$.
We will focus on dimensions $n\ge 3$ in order to take advantage of the full
range of exponents $p \in [2,\infty]$ in~\eqref{eq:Strichartz}.

The Schr\"odinger propagator of $H$ generally fails to satisfy estimates 
like~\eqref{eq:Strichartz} due to the possible existence of bound
states, quasiperiodic solutions obeying $u(t + 2\pi,x) = e^{2\pi i\lambda}u(t,x)$
for all $t,x \in \R^{1+n}$ and possessing moderate spatial decay.
These are best understood in terms of the Floquet Hamiltonian
\begin{equation}
K = i\partial_t - \Delta_x + V(t,x)
\end{equation}
acting on $2\pi$-periodic functions with domain $\T \times \R^n$.
Precisely, each bound state $\phi(t,x)$ solves the distributional equation
$(K - \lambda)e^{-i\lambda t}\phi = 0$.  If $e^{-i\lambda t}\phi$ is
time-periodic and belongs to the space $L^2(\T \times \R^n)$ then it is an
eigenfunction of $K$ with eigenvalue $\lambda$.
We say that $K$ has a {\em resonance} at $\lambda$ if the resolvent
$(K-(\lambda\pm i0))^{-1}$ is singular but the associated 
``eigenfunction'' is not square-integrable.  The precise definition is 
postponed until Section~\ref{sec:Inverses}, where we attempt to estimate
the resolvent of $K$ in the neighborhood of singularities. 
The spectrum of $K$ is invariant under integer shifts, as
$(K - n) = e^{-int}K e^{i nt}$ for any $n \in \Z$.


Because our assumptions do not imply that $K$ is self-adjoint, the spectrum of
$K$ need not be confined to the real axis.  Each eigenfunction
$\phi_\lambda$ with $\lambda \not\in \R$ illustrates the related 
lack of an $L^2$ conservation law for solutions of the Schr\"odinger equation.
Since $|e^{2\pi i\lambda}| = e^{-2\pi \Im[\lambda]} \not= 1$, the $L^2$ norm of
$\phi_\lambda$ decreases exponentially in one time direction and grows in the
other.

For each $\lambda \in \Compl$ define $N_\lambda$ to be the solution space
\begin{equation} \label{eq:Nlambda}
N_\lambda = \{\phi: (K-\lambda)e^{-i\lambda t}\phi = 0,\, e^{-i\lambda t}\phi
 \in L^2(\T\times\R^n)\}
\end{equation}
Local regularity theory dictates that every true eigenfunction also satisfies
$e^{-i\lambda t}\phi \in C(\T; L^2(\R^n))$.  It is then permissible to discuss
the initial value of an eigenfunction, $\Phi = \phi(0,\,\cdot\,)$.  The
projection of $N_\lambda$ onto the space of initial data has as its image
\begin{equation} \label{eq:Xlambda}
X_\lambda = \{\Phi: \phi \in N_\lambda\} \subset L^2(\R^n).
\end{equation}
We will show via a compactness argument that both
$N_\lambda$ and $X_\lambda$ are always finite dimensional.
Similarly define $\tilde{N}_\lambda$ and $\tilde{X}_\lambda$ to represent the
eigenfunctions of $\overline{K}$ (the Floquet operator with potential
$\overline{V(t,x)}$) that have eigenvalue $\bar{\lambda}$.  These spaces are
all invariant under real integer translations.

In this paper we prove that the Schr\"odinger evolution of $H = -\Delta +
V(t,x)$ observes a space-time estimate identical to~\eqref{eq:Strichartz}
once a finite-dimensional space of bound states are projected away. 
Our primary assumptions are that $V(t,x)$ be periodic and belong to the
scaling-invariant space $L^{n/2}_x L^\infty_t$ and that each of the
bound states is an eigenfunction of sufficient decay and/or regularity.
If we further assume that $V$ is real-valued with polynomial pointwise decay
and some smoothness with respect to $t$,
then only the bound states at $\lambda \in \Z$ are a concern, and only in 
dimensions $n \le 6$.  Improvements of this type are discussed immediately
following our statement of the main theorem.

\begin{theorem} \label{thm:main}
Let $V(t,x)$ be a time-periodic function on $\R^{1+n}$, $n \ge 3$, satisfying
$V(t+2\pi,x) = V(t,x)$ at almost every $t,x$ and belonging to the class
$L^{n/2}_x L^\infty_t$.  Suppose that $K$ and $\overline{K}$ have no
resonances along the real axis, and that their behavior at each
eigenvalue $\lambda \in \Compl$ satisfies the conditions
\renewcommand{\theenumi}{(C\arabic{enumi})}
\renewcommand{\labelenumi}{\theenumi}
\begin{enumerate}
\item \label{C1}
 $e^{-i\lambda t}N_\lambda$ and $e^{-i\bar{\lambda} t}\tilde{N}_\lambda$
are both contained in
 $L^{\frac{2n}{n+2}}(\R^n; L^2(\T)) + \japanese[x]^{-1} L^2(\T \times \R^n)$,
\item \label{C2}
 $X_{\lambda}$ and $\tilde{X}_\lambda$ are subspaces of
 $\japanese[x]^{-1}L^2(\R^n) + W^{1,\frac{2n}{n+2}}(\R^n)$,
\item \label{C3}
The $L^2$-orthogonal projection of $X_\lambda$ onto $\tilde{X}_\lambda$
is bijective.
\end{enumerate}
 
Under these assumptions, there exist at most finitely many eigenvalues of
$K$, $\overline{K}$ in the strip $\Compl/\Z$, counted with multiplicity.
Furthermore, the initial value problem for the Schr\"odinger equation
\begin{equation}\label{eq:Schr}
\begin{cases} (i\partial_t -\Delta_x  + V(t,x))&u(t,x) = 0, \quad  
  x\in \R^n, t\in \R  \\
    &u(0,x) = f(x), \quad x\in \R^n \end{cases}
\end{equation}
possesses a unique weak solution in the Strichartz space 
$L^2_tL^{2n/(n-2)}_x$, satisfying
\begin{equation} \label{eq:newStrich}
\norm[u][L^2_tL^{2n/(n-2)}_x] + \norm[u][C_b(\R; L^2(\R^n))]\ \les\ \norm[f][2]
\end{equation}
for all initial data $f$ in the $L^2$-orthogonal complement of
$\tilde{X} = \oplus_\lambda \tilde{X}_\lambda$.
\end{theorem}

\begin{remark}
In the general case, where $K$ is not self-adjoint, the conclusion that
$u \in C_b(\R; L^2(\R^n))$ for most initial data is a nontrivial
$L^2$-stability result.
\end{remark}

\begin{remark}
If $V$ is real-valued, then each eigenvalue $\lambda$ is also real. 
Since $\overline{K} = K$, it also follows
that $\tilde{N}_\lambda = N_\lambda$ and $\tilde{X}_\lambda = X_\lambda$,
making the condition~\ref{C3} unnecessary.
\end{remark}


\begin{remark} 
The unweighted portion of condition~\ref{C2} is not sharp in terms
of the number of derivatives required. Lemma~\ref{lem:discreteKato} and its
supporting propositions construct a family of lower-regularity spaces which may
be used in place of $W^{1,2n/(n+2)}(\R^n)$.
\end{remark}

\begin{corollary}
Suppose that the time-periodic potential $V(t,x)$ is real valued and satisfies
the bound
\begin{equation} \label{eq:extra}
\sup_{x\in\R^n} \japanese[x]^\beta \norm[V(\,\cdot\,,x)][H^s(\T)] < \infty
\end{equation}
for some $\beta > 2$ and $s > \frac12$.  The Strichartz estimates in 
Theorem~\ref{thm:main} are valid provided that $\lambda \in \Z$ is not a 
resonance, and any eigenvectors at $\lambda \in \Z$ belong to 
$\japanese[x]^{-1}L^2$.

In dimensions $n \ge 7$, Theorem~\ref{thm:main} is valid for all real-valued
potentials satisfying~\eqref{eq:extra}.  No further conditions are
necessary.
\end{corollary}
\begin{proof}
Due to the self-adjointness of $K$, there are no eigenvalues off of the real
axis.  Following the proof of Lemma~2.8 in~\cite{GJY}, resonances can
only exist at $\lambda \in \Z$, and if $\lambda$ is not an integer then the
eigenfunctions additionally satisfy $\phi_\lambda \in \japanese[x]^{-N}
H^s(\T; L^2(\R^n))$.  The main ingredients are an Agmon-type bootstrapping
argument (based on~\cite{Agmon}) and the fact that multiplication by a
function in $H^s(\T)$ preserves the $H^{s-1/2}(\T)$ norm.

When $\lambda \in \Z$, the bootstrapping process produces only as much spatial
decay for $\phi_\lambda$ as is present in the Green's function of the
Laplacian.  In general, the Green's function belongs to $\japanese[x]^\sigma
L^2$ (aside from the local singularity) for all $\sigma > \frac{4-n}{2}$.
For $n \ge 7$, the desired value $\sigma = -1$ is part of this range.
\end{proof}

\begin{corollary}
Let $V(x) \in L^{\frac{n}{2}}(\R^n)$ be a complex valued time-independent 
potential.  The Strichartz estimates in Theorem~\ref{thm:main} are valid 
provided the equation
\begin{equation*}(-\Delta + V - \lambda)\phi = 0
\end{equation*}
has no solutions $\phi \in L^{2n/(n-2)}(\R^n)$ for any $\lambda \in [0,\infty)
\subset \Compl$, and condition~\ref{C3} is satisfied at every eigenvalue.
\end{corollary}
\begin{proof}
Similar to the preceding corollary, the point is that all of the permitted
bound states $\phi_\lambda = e^{i\lambda t}\Phi(x)$
are necessarily eigenfunctions that decay rapidly enough to
satisfy condition~\ref{C2}.  In this case the bootstrapping is based on the
relation $\Phi = -(I + (-\Delta -\lambda)^{-1}V)\Phi$ 
Since $\lambda \not\in [0,\infty)$, the resolvent of the Laplacian is bounded
from every $L^p(\R^n)$ to itself, $1 \le p \le \infty$.

Starting with $\Phi \in L^{2n/(n-2)}$, one iteration brings the exponent down
to $\Phi \in L^{2n/(n+2)}$.  Furthermore it is quite easy to take two
derivatives:  $\Delta \Phi = V\Phi - \lambda\Phi \in L^{2n/(n+2)}$.
Thus $\phi \in W^{1,2n/(n+2)}$ as is required by~\ref{C2}.
\end{proof}
\begin{corollary} \label{cor:RealTimeInd}
If $V \in L^{\frac{n}{2}}(\R^n)$ is a real-valued potential,
then~\eqref{eq:newStrich}
holds provided the Schr\"odinger operator $H = -\Delta + V$ does not have a
resonance or an eigenvalue at zero energy.
\end{corollary}
\begin{proof}
In this case the spectrum of $H$ is purely absolutely continuous on the
interval $(0,\infty)$ due to the combined results of~\cite{GS2} 
and~\cite{IonJer}. According to the previous corollary, the only remaining
spectral point of concern is the behavior of $H$ at $\lambda = 0$. 
The additional assumption ensures that zero is a regular point of the
spectrum as well.
\end{proof}

Although Theorem~\ref{thm:main} is presented as a perturbation
of the Strichartz inequality~\eqref{eq:Strichartz}, which in turn is based
on dispersive estimates for the free Schr\"odinger evolution,
we do not attempt to prove comparable dispersive estimates for $H$.  This is
partly a matter of convenience, as the study of time-asymptotics for
Floquet operators (as in~\cite{GJY}) presents its own set of technical
challenges.  More importantly, the conditions for Theorem~\ref{thm:main} 
include numerous potentials for which the corresponding dispersive
estimate are known to fail.

The discrepancy is especially apparent in dimensions $n\ge 4$.  No pointwise
or $L^p$ condition on the potential is sufficient by itself to imply an
$L^1 \to L^\infty$ dispersive bound~\cite{GV}.  Either some extra regularity
of $V$ is needed, as in~\cite{JSS}, or one must expect to suffer a loss
of derivatives in the solution~\cite{Vodev}.  On the other hand, Strichartz
estimates were proven in~\cite{RS} for time-independent potentials satisfying
$|V(x)| \les \japanese[x]^{-2-\eps}$.  In this work the authors used Kato
smoothing estimates as the intermediary step in place of the nonexistent
dispersive bounds.  Corollary~\ref{cor:RealTimeInd} represents a modest
extension of this work.

We wish to emphasize one additional feature of Theorem~\ref{thm:main} that
appears to be unique in the literature: the treatment of eigenvalues
depends only on the nature of the associated eigenfunction, not on its
location relative to the spectrum of $K$.  While it may be true in certain
applications that threshold eigenvalues and/or resonances enjoy distinct
properties from those embedded in the continuous spectrum or from isolated
points, the criteria~\ref{C1}-\ref{C3} apply equally in all these
cases.

\smallskip
The proof of Theorem~\ref{thm:main} is based on a direct application of
Duhamel's formula.  We consider
the behavior of solutions when $t\ge 0$; the reasoning for $t \le 0$ is 
identical.  Let $\Uplus$
denote the forward propagator of the free Schr\"odinger equation, that is
\begin{equation*}
\Uplus g(t,x) := \int_{s<t} e^{-i(t-s)\Delta}g(s,x)\,ds.
\end{equation*}
We will also allow $\Uplus$ to act on functions of $x$ alone by the definition
$\Uplus g(t,x) := \chi_{[0,\infty)}(t)e^{-it\Delta}g(x)$.  The adjoint of 
$\Uplus$ in both cases is the backward propagator $\Uminus$. The full range of 
mapping properties of $\Uplus$ are established in~\cite{KT}; of particular
concern are the bounds
\begin{equation} \label{eq:KeelTao}
\begin{alignedat}{3}
\Uplus &:\, &L^2_t &L^{2n/(n+2)}_x && \to L^2_t L^{2n/(n-2)}_x
               \cap C(\R; L^2_x) \\
\Uplus &:\, &&L^1_t L^2_x&& \to L^2_t L^{2n/(n-2)}_x
               \cap C(\R; L^2_x) \\
\Uplus &:\, &&L^2_x &&\to L^2_t L^{2n/(n-2)}_x \cap C([0,\infty); L^2_x)
\end{alignedat}
\end{equation}

Every weak solution of \eqref{eq:Schr} on the time interval $[0,\infty)$
must solve the functional equation $u(t,x) = \Uplus f(t,x) +i \Uplus Vu(t,x)$.
This leads to the formal solution
\begin{equation*}
u = (I - i\Uplus V)^{-1}\Uplus f
\end{equation*}
where the inverse is taken among bounded operators on 
$L^2_t L^{2n/(n-2)}_x$.  In order to work in the setting of $L^2_tL^2_x$,
factorize $V = ZW$, with $Z, W \in L^\infty_t L^n_x$ and write
\begin{equation} \label{eq:Duhamel}
u = \Uplus f\  +\ i\Uplus Z(I -iW\Uplus Z)^{-1}W\Uplus f.
\end{equation}
In the event that $I - iW\Uplus Z$ is invertible as an operator on 
$L^2([0,\infty); L^2(\R^n))$, one concludes that~\eqref{eq:newStrich} holds
for all $f \in L^2$
which implies an absence of bound states.  This occurs for all
$V \in L^\infty_t L^{n/2}_x$ of sufficiently small norm.
In every other case, the challenge is to find a condition on $f$ so that 
$W\Uplus f$ belongs to the domain of the unbounded operator
$(I - iW\Uplus Z)^{-1}$.

Much of our analysis is done with respect to the Fourier transform of the
time variable, in deference to the fact that $\Uplus$ and $V$ preserve
the space of functions satisfying $g(t+2\pi,x) = e^{2\pi i\lambda}g(t,x)$
for each $\lambda \in [0,1]$.  We show that $I - iW\Uplus Z$ is a compact 
perturbation of the identity on each of these spaces.  The Fredholm
Alternative then equates invertibility with the absence of eigenvalues or
resonances at $\lambda$.

Common sense suggests that the singularities caused by a particular bound 
state $\phi$ can be avoided by requiring the initial data $f$ to be
orthogonal to $\Phi$.  Even in the time-independent case, however, 
eigenvalues at zero energy are known to disturb dispersive estimates
after such a projection.  This phenomenon is first identified 
in~\cite{JK} and described in more detail in~\cite{ES}.
A full asymptotic expansion for Floquet solutions has recently been computed
in three dimensions in~\cite{GJY}.  We note that the intuitive suggestion
above is also incorrect when the Schr\"odinger propagation is not unitary
(i.e. when $K$ has complex values).  The projection employed in
Theorem~\ref{thm:main} is actually orthogonal to a function 
$\tilde{\Phi} \in \tilde{X}$ rather than $\Phi$.  

In order to determine the success of a projection, we closely 
examine the behavior of $(I - iW\Uplus Z)^{-1}$ for all $\lambda$ in the
neighborhood of an eigenvalue and assess whether it is compatible with the
input $W\Uplus f$.  The resulting eigenvalue condition appears in the form of
a discrete-time Kato smoothing bound.  This last computation, parts of which
are adapted from~\cite{KRS} and~\cite{Simon}, may be of independent interest.

\section{Resolvents, Compactness, and Continuity}

In this section we aim to find spaces on which $I - iW\Uplus Z$ is a compact
perturbation of the identity.  
For each $\lambda \in \Compl$,
define
\begin{equation*}
Y_\lambda = \big\{g \in L^{2,loc}_t L^2_x : g(t+2\pi,x) = 
e^{2\pi i\lambda}g(t,x)\big\}.
\end{equation*}
Functions $g \in Y_\lambda$ are naturally associated with the periodic 
$e^{-it\lambda}g \in L^2(\T\times\R^n)$, and we use this identification
to define a Hilbert space norm on $Y_\lambda$.

For each $\lambda \in \R/\Z$, there exists a ``projection'' $P_\lambda$
from $L^2_t L^2_x$ onto $Y_\lambda$ given by
\begin{equation*}
P_\lambda g(t,x) = \sum_{m\in\Z} e^{-2\pi i\lambda m}g(t+2\pi m,x)
\end{equation*}
The family of operators $P_\lambda$ can be understood as a partial Fourier
transform in the time variable, acting on the space $L^2_t L^2_x \cong
\ell^2_m (L^2([2\pi m, 2\pi (m+1)]; L^2(\R^n)))$.  For example the Plancherel
identity is expressed as
\begin{equation} \label{eq:Plancherel}
\int_0^1 \norm[P_\lambda g][Y_\lambda]^2\,d\lambda = \norm[g][L^2_t L^2_x]^2
\end{equation}
For functions $g$ with support in the halfline $t \in [0,\infty)$, 
the definition of $P_{\lambda}g$ extends to the strip
$\lambda = \lambda' + i\mu$, $\mu \le 0$, $\lambda' \in \R/\Z$ with the value 
$e^{-\mu t}P_{\lambda'}(e^{\mu t}g)$.  
The Plancherel identity in this case becomes
\begin{equation*}
\int_0^1 \norm[P_{\lambda'+i\mu} g][Y_{\lambda'+i\mu}]^2\, d\lambda' =
\int_0^1 \norm[P_{\lambda'} e^{\mu t}g][Y_{\lambda'}]^2\,d\lambda' =
\norm[e^{\mu t}g][L^2_t L^2_x]^2
\end{equation*}

On the Fourier side with respect to time, $P_\lambda$ has a very clear
interpretation.  Let $\hat{g}(\tau,x)$ be the partial Fourier transform of
$g$.  Then $(P_\lambda g)\hat{\,}$ is the restriction of $\hat{g}$ to the
planes $\{\tau \in \lambda + \Z\}$.  If $g$ is supported on $\{t \ge 0\}$
then $\hat{g}$ has an analytic extension to the lower halfplane, making the
restrictions to $\{\tau \in \lambda' + i\mu+ \Z\}$ well-defined.
Clearly $P_\lambda$ commutes with pointwise multiplication (in $(t,x)$) by
any $2\pi$-periodic function.

The action of $\Uplus$ in this setting is also easy to characterize.
Since $\Uplus$ convolves functions in the time variable with the integral
kernel $K(t) = \lim_{\eps\downarrow 0} e^{-it\Delta -\eps t}\chi_{t\ge 0}$,
on the Fourier side it performs pointwise (in $\tau$) ``multiplication''
by $\hat{K}(\tau) = \lim_{\eps\downarrow 0} i(-\Delta - (\tau-i\eps))^{-1}$.
Using the notation of resolvents,
\begin{equation} \label{eq:UplusRminus}
(\Uplus g)\hat{\,}(\tau,x) = i\Rminus(\tau)\hat{g}(\tau,x)
\end{equation}
where $\Rminus(\tau)$ represents the branch of the resolvent of $-\Delta$
which continues analytically to $\{\Im[\tau] \le 0\}$.  
This shows that $\Uplus$ also commutes with each of the projections
$P_\lambda$.  Once again, if $\supp_t g \subset [0,\infty)$, 
the identity~\eqref{eq:UplusRminus} remains valid for all $\tau$ in the lower
halfplane, with the understanding that
\begin{equation*}
\hat{g}(\tau,x) = (e^{\Im[\tau]}g)\hat{\,}(\Re[\tau],x).
\end{equation*}
Therefore the operator
$I-i W\Uplus Z$ admits a restriction to each $Y_\lambda$, 
$\Im[\lambda] < 0$, and most importantly,
\begin{equation} \label{eq:goal}
\norm[e^{\mu t}(I - iW\Uplus Z)^{-1}W\Uplus f][L^2_tL^2_x]^2 = 
\int_0^1 
\norm[(I - iW\Uplus Z)^{-1}P_{\lambda'+i\mu} W\Uplus f][Y_{\lambda'+i\mu}]^2
  \, d\lambda'
\end{equation}
The proof of Theorem~\ref{thm:main} will be complete once we bound this
quantity in terms of the $L^2(\R^n)$ norm of $f$, uniformly over $\mu \le 0$.

The particular factorization we choose for $V(t,x)$ is to let
$W(t,x) = w(x) = (\norm[V(\,\cdot\,, x)][\infty])^{\frac12}$.  By our
assumptions, $w \in L^n(\R^n)$.  The remaining factor can be decomposed as
$w(x)z(t,x)$, with $w$ the same function as above and $z(t,x)$ periodic and
bounded almost everywhere by 1.  Multiplication by $z$ is a bounded operator
of unit norm on $Y_\lambda$, so compactness of the operator $w\Uplus wz$
follows directly from compactness of $w\Uplus w$.
 
\begin{proposition} \label{prop:resolvent}
Given any function $w \in L^n$, the collection 
$\{w\Rminus(\tau)w: \Im[\tau] \le 0\}$ forms a uniformly
continuous family of compact operators on $L^2(\R^n)$ with norm decreasing
to zero as $|\tau| \to \infty$.
\end{proposition}
\begin{proof}
This is a compilation of well-known resolvent estimates, primarily the fact
(proved in~\cite{KRS}) that $\Rminus(\tau)$ are uniformly bounded as operators
from $L^{\frac{2n}{n+2}}$ to $L^{\frac{2n}{n-2}}$.  All of the desired
properties -- compactness, continuity, and norm decay -- are preserved if
$w$ is approximated in $L^n$ by a sequence of bounded compactly supported
functions $w^{\eps}$.

For compactness, observe that $(-\Delta + 1)\Rminus(\tau)w^{\eps}g
= w^{\eps}g + (\tau+1) \Rminus(\tau)w^{\eps}g \in L^2(\R^n)$.  Within any
ball of finite radius $R$, the Sobolev space $H^2$ embeds compactly inside
$L^{\frac{2n}{n-2}}$.  If this ball is much larger than the support of 
$w^{\eps}$, then there is a pointwise bound 
\begin{equation*}
|\Rminus(\tau)w^{\eps}g(x)|
\les |\tau|^{\frac{n-3}{4}} \norm[g][2] \norm[w^{\eps}][2] |x|^{\frac{1-n}{2}}
\end{equation*}
outside of the ball.  Allowing $R \to \infty$ expresses 
$w^{\eps}\Rminus(\tau)w^{\eps}$ as a norm-limit of compact operators on $L^2$.

For continuity, recall that the integration kernel of $\Rminus(\tau)$ is
$|x-y|^{2-n}F(\tau^{\frac12}|x|)$, where $F$ can be expressed explicitly in
terms of Hankel functions.  In dimensions $n \ge 3$ it satisfies the 
pointwise bounds
\begin{equation*}
|F(z)|, |F'(z)| \les \japanese[z]^{(n-3)/2}.
\end{equation*}
Using the mean value theorem, if $|\tau-\sigma| < \frac12|\tau|$ then
\[
|\Rminus(\tau)-\Rminus(\sigma)(x,y)| \les 
\begin{cases} |\tau^{\frac12}-\sigma^{\frac12}|\,|x-y|^{3-n},\ 
    & {\rm if}\ |x-y| < |\tau|^{-\frac12} \\
 |\tau|^{\frac{n-3}{4}} |\tau^{\frac12}-\sigma^{\frac12}|
  |x-y|^{\frac{3-n}2}, \ &{\rm if}\ 
|\tau|^{-\frac12} < |x-y| < |\tau^{\frac12}-\sigma^{\frac12}|^{-1}
\end{cases}
\]
The case where $|x-y|$ is large is unimportant because $w^{\eps}$ has compact 
support.  The Schur test then shows that $w^{\eps}\Rminus(\tau)w^{\eps}$
is continuous with respect to $\tau$.

Finally, decay as $|\tau| \to \infty$ is an immediate consequence of another
resolvent bound from~\cite{KRS}, namely that 
$|\tau|^{\frac{1}{n+1}}\Rminus(\tau)$ is a uniformly bounded family of maps
from $L^{\frac{2n+2}{n+3}}$ to $L^{\frac{2n+2}{n-1}}$.  The combination of
continuity and decay at infinity immediately implies uniform continuity.
\end{proof}

\begin{corollary} \label{cor:compact}
Given any $w \in L^n(\R^n)$, the collection 
$\{ e^{-i\lambda t} w \Uplus w e^{i\lambda t}: \Im[\lambda] \le 0\}$ 
is a continuous family (with respect to $\lambda$)
of compact operators on $L^2(\T\times \R^n)$, with norm decreasing to
zero as $\Im[\lambda] \to -\infty$.

The same is also true for the family of operators
$e^{-i\lambda t} w \Uplus wz e^{i\lambda t}$ for any bounded
$2\pi$-periodic function $z$.
\end{corollary}

\begin{proof}
For every choice of $\lambda$ in the lower halfplane, 
the Fourier series coefficients of
$e^{-i\lambda t}w \Uplus w e^{+i\lambda t}g$ are precisely 
$\{w\Rminus(\lambda+k)w\hat{g}(k,x): k \in \Z\}$.  At each $k$
this is a compact operator on $\R^n$, and the norms decrease as 
$|k| \to \infty$. It follows that their collective action on 
$\ell^2(k; L^2(\R^n))$ is a compact operator with norm
$\sup_k \norm[w\Rminus(\lambda+k)w][]$.  As $\Im[\lambda] \to -\infty$,
the norm is bounded by 
\begin{equation*}
\sup_{|\tau| > |\Im[\lambda]|} \norm[w\Rminus(\tau)w][]
\end{equation*} 
which decreases to zero.

Given two numbers $\lambda_1$ and $\lambda_2$, the norm
difference of their associated operators is
\[\sup_k \norm[w(\Rminus(\lambda_1+k)-\Rminus(\lambda_2+k))w][].\]
The uniform continuity assertion in Proposition~\ref{prop:resolvent}
takes this to zero in the limit $\lambda_2 \to \lambda_1$.

Neither the compactness nor continuity properties of $e^{-i\lambda t}w\Uplus w
e^{i\lambda t}$ are affected by composition
with the bounded operator $e^{-i\lambda t}ze^{i\lambda t}$.
\end{proof}

\section{Estimates for Inverse Operators} \label{sec:Inverses}


There are two main elements in the expression~\eqref{eq:goal}, an inverse
operator $(I -iw\Uplus wz)^{-1}$ and a series of functions 
$P_\lambda w\Uplus f \in Y_\lambda$.  In this section we prove uniform
bounds for $(I - iw\Uplus wz)^{-1}$ on $Y_\lambda$ where possible,
and describe the singularities that occur as $\lambda$ approaches
the spectrum of $K$.

The spaces $Y_\lambda$ are a natural setting for working with 
bound states, especially those bound states that grow exponentially over time.
When we wish to vary $\lambda$ as a parameter, however, a unified approach
based on $L^2(\T\times\R^n)$ is preferred.  Define the family of operators
\begin{equation*}
T(\lambda) = I -ie^{-i\lambda t}w\Uplus wze^{i\lambda t} 
 = I -iw(e^{-i\lambda t}\Uplus e^{i\lambda t})wz
\end{equation*}
acting on $L^2(\T\times\R^n)$.  The kernel of $T(\lambda)$ provides valuable
information about the spectrum of $K$, thanks to the intertwining relations
\begin{align*}
 (K-\lambda)\big(e^{-i\lambda t}\Uplus e^{i\lambda t}wz\big)
 &= i wz T(\lambda), \\
 \big(we^{-i\lambda t}\Uplus e^{i\lambda t}\big)(K-\lambda)
 &= i T(\lambda) w.
\end{align*}
Each element
$g \in \ker T(\lambda)$ corresponds to a bound state 
$\phi = \Uplus wz e^{i\lambda t}g$.  
Proposition~\ref{prop:Uplus} below shows that $e^{-i\lambda t}\phi$ is a true
eigenfunction of $K$ in $L^2(\T\times\R^n)$ if $\Im[\lambda] < 0$.
Additional tools are available (\cite{GJY},~\cite{Y1}) if $V$ is real-valued
and $\lambda \not\in\Z$.  In any of the remaining cases it is possible that
the spatial decay of $\phi$ fails to be square-integrable.  We say that
$K$ has a {\em resonance} at $\lambda$ when this occurs; that is, when
there exists some $g \in \ker T(\lambda)$ for which 
$\phi = \Uplus wz e^{i\lambda t}g$ does not belong to $L^2(\T\times\R^n)$.


Note that $T(\lambda+1)$ is a unitary
conjugate of $T(\lambda)$, so one only needs to check the invertibility of
$T(\lambda)$ inside the strip 
\begin{equation*}
\Omega^- = 
\{\lambda \in \Compl: \Re[\lambda] \in [0,1),\, \Im[\lambda] \le 0 \}.  
\end{equation*}
The set $\Omega^- \subset \Compl$ is a fundamental domain for the lower
halfplane modulo the integers, and will always be given the quotient topology.
We make some remarks about the size and
differentiability properties of $e^{-\lambda t}\Uplus e^{i\lambda}$ for
future reference.

\begin{proposition} \label{prop:Uplus}
For each $\lambda$ with $\Im[\lambda] < 0$, the operator 
$e^{-i\lambda t}\Uplus e^{i\lambda t}$ is subject to the following estimates.
\begin{align} \label{eq:UplusL2}
\norm[e^{-i\lambda t}\Uplus e^{i\lambda t}g][L^2(\T\times\R^n)]
  &\les |\Im[\lambda]|^{-1}\norm[g][L^2(\T\times\R^n)] \\
\norm[e^{-i\lambda t}\Uplus e^{i\lambda t}g][L^2(\T\times\R^n)]
  &\les |\Im[\lambda]|^{-\frac12}\norm[g][L^2(\T; L^{2n/(n+2)}(\R^n))]
 \label{eq:UplusLp}
\end{align}
Given two values $\lambda_1, \lambda_2$, the difference can be expressed as
\begin{equation} \label{eq:Uplusdiff}
e^{-i\lambda_1 t}\Uplus e^{i\lambda_1 t} - 
 e^{-i\lambda_2 t}\Uplus e^{i\lambda_2 t} = -i(\lambda_1 - \lambda_2)\big(
 e^{-i\lambda_1 t}\Uplus e^{i\lambda_1 t}\big)\big(e^{-i\lambda_2 t}\Uplus
 e^{i\lambda_2 t}\big).
\end{equation}
Therefore the family of operators $e^{-i\lambda t}\Uplus e^{i\lambda t}$
possesses the holomorphic derivative
\begin{equation} \label{eq:Uplusderiv}
\frac{d}{d\lambda}\big[e^{-i\lambda t}\Uplus e^{i\lambda t}\big]
 = -ie^{-i\lambda t}(\Uplus)^2 e^{i\lambda t}.
\end{equation}
over the domain $\Im[\lambda] < 0$.
\end{proposition}
\begin{proof}
The estimates \eqref{eq:UplusL2} and \eqref{eq:UplusLp} both exploit the
facts that $\Uplus g(t,x)$ depends only on $\chi_{s < t}g(s,x)$, and that
$e^{i\lambda t}$ decays exponentially as $t \to -\infty$.
To be precise, if $g \in L^2(\T\times\R^n)$, then the $L^1_tL^2_x$ norm of
$\chi_{(-\infty, t)} e^{i\lambda s}g$ is bounded by 
$|\Im[\lambda]|^{-1}e^{-\Im[\lambda] t}$. Similarly, if 
$g \in L^2(\T; L^{\frac{2n}{n+2}}(\R^n))$ then the
$L^2_tL^{2n/(n+2)}_x$ norm of $\chi_{(-\infty, t)} e^{i\lambda s}g$
is bounded by $|\Im[\lambda]|^{-\frac12}e^{-\Im[\lambda]t}$.  
In either case the propagator 
estimates~\eqref{eq:KeelTao} complete the argument.

The difference and derivative formulas can be verified directly, or by
expressing $\Uplus$ according to its Fourier
representation~\eqref{eq:UplusRminus}.
The equivalent identities for resolvents are $\Rminus(\lambda_1) -
\Rminus(\lambda_2) = (\lambda_1 - \lambda_2)\Rminus(\lambda_1)
 \Rminus(\lambda_2)$ and 
$\frac{d}{d\lambda}\Rminus(\lambda) = \Rminus(\lambda)^2$.
\end{proof}

Corollary~\ref{cor:compact} shows that
each $T(\lambda)$, $\Im[\lambda] \le 0$, is a compact perturbation of the
identity.  Furthermore, $\norm[T(\lambda)^{-1}][]$ varies continuously over
its domain of definition, is periodic with respect to translation by $\Z$, 
and is bounded by $2$ once the imaginary part of
$\lambda$ is sufficiently negative.  If $T(\lambda)^{-1}$ existed everywhere,
this would suffice to bound its norm uniformly in $\lambda$.  By the
Fredholm Alternative, only an eigenvalue or resonance at $\lambda$ can
prevent $T(\lambda)$ from being invertible.  We examine the structure of
these singularities in the following two lemmas.

\begin{lemma} \label{lem:inverse}
Let $w \in L^n(\R^n)$ and $z \in L^\infty(\T\times \R^n)$.  Suppose the
operator $T(\lambda_0)$ fails to be invertible for some 
$\lambda_0 \in \Compl$ with $\Im[\lambda_0] < 0$.
Then the solution spaces $N_{\lambda_0} \subset Y_{\lambda_0}$ and 
$\tilde{N}_{\lambda_0} \subset Y_{\bar{\lambda}_0}$
are both nontrivial and finite dimensional.  The set of their initial values,
$X_{\lambda_0}$ and $\tilde{X}_{\lambda_0}$,
are well defined finite dimensional subspaces of $L^2(\R^n)$.

If the orthogonal projection from $X_{\lambda_0}$ onto $\tilde{X}_{\lambda_0}$
is bijective, then $T(\lambda)$ is invertible for every
other $\lambda$ in a neighborhood of $\lambda_0$.  More precisely,
\begin{equation} \label{eq:inverse}
\norm[T(\lambda)^{-1}
     (h_1 + h_2)][L^2(\T\times\R^n)] 
\le C(w,z,\lambda_0)(|\lambda - \lambda_0|^{-1}\norm[h_1][]+ \norm[h_2][])
\end{equation}
where $h_1 = e^{-i\bar{\lambda}_0 t}\overline{zw}\tilde{\phi}$, 
$\tilde{\phi} \in \tilde{N}_{\lambda_0}$,  
and $h_2$ belongs to the $L^2$-orthogonal complement of
$e^{-i\bar{\lambda}_0 t}\overline{zw}\tilde{N}_{\lambda_0}$.
\end{lemma}

\begin{proof}

The operator $T(\lambda_0)$ is a compact perturbation of the identity, and by
assumption it is not invertible.  The Fredholm Alternative asserts that 
$T(\lambda_0)$ has a finite dimensional kernel,
a cokernel of the same dimension, and that it is an invertible map between
their respective orthogonal complements.

Every element $g \in L^2(\T\times\R^n)$ in the kernel of $T(\lambda_0)$ is 
associated to a prospective eigenfunction $e^{-i\lambda_0 t}\phi$ by the
relations $\phi = \Uplus wz e^{i\lambda_0 t}g$ and
$g = ie^{-i\lambda_0 t}w\phi$. 
Note that $wzg \in L^2(\T; L^{\frac{2n}{n+2}}(\R^n))$, so the
mapping estimate~\eqref{eq:UplusLp} implies that $e^{-\lambda_0 t}\phi$ 
belongs to $L^2(\T\times\R^n)$.  That makes $e^{-i\lambda_0 t}\phi$ an
eigenfunction of $K$, and $\phi \in N_{\lambda_0}$ by definition.
It follows immediately that
\begin{equation*}
\ker T(\lambda_0) = e^{-i\lambda_0 t}w N_{\lambda_0}.
\end{equation*}
In general, a function $\phi \in L^2(\T\times\R^n)$ should not have a
meaningful initial value $\Phi(x) = \phi(0,x)$.  On the other hand,
$\phi$ solves the inhomogeneous Schr\"odinger equation
\begin{equation*}
(i\partial_t -\Delta)\phi = -V\phi \in L^{2, loc}_t L^{2n/(n+2)}_x
\end{equation*}
from which Duhamel's formula (averaged over all starting times
$s \in [-2\pi,0]$) yields
\begin{equation}
\begin{aligned}
\phi(0,x) &= (2\pi)^{-1}\int_{-2\pi}^0 \bigg(e^{i\Delta s}\phi(s,x)
 +i \int_s^0 e^{i\Delta r}V\phi(r,x)\,dr \bigg)\,ds \\
 &= (2\pi)^{-1} \bigg(\int_{-2\pi}^0 e^{i\Delta s}\phi(s,x)\,ds
  +i \int_{-2\pi}^0 e^{i\Delta r}(r+2\pi)V\psi(r,x)\, dr \bigg)
\end{aligned}
\end{equation}
The first integral evaluates to a function in $L^2(\R^n)$ because
$e^{i\Delta s}$ is unitary and $\phi \in L^{1,loc}_tL^2_x$.  The second
integral does likewise, via the dual statement of~\eqref{eq:KeelTao}.

\begin{remark}
Because $\ker T(\lambda_0)$ is a finite dimensional space, the norms of
$g \in \ker T(\lambda_0) \subset L^2(\T\times\R^n)$ and $\phi \in
N_{\lambda_0} \subset Y_{\lambda_0}$ are equivalent.  These norms are also
equivalent to the norm of $\Phi \in X_{\lambda_0} \subset L^2(\R^n)$
for the same reason.
\end{remark}

The image of $T(\lambda_0)$ consists of all
functions orthogonal to the kernel of its adjoint, namely
\begin{equation*}
T(\lambda_0)^* = I + ie^{-i\bar{\lambda_0}t}\overline{zw}\Uminus\overline{w}
 e^{i\bar{\lambda}_0 t}.
\end{equation*}  
Every element $\tilde{g}$ in the kernel of $T(\lambda_0)^*$
is associated to an eigenfunction
$e^{-i\bar{\lambda}_0 t}\tilde{\phi} \in \tilde{N}_{\lambda_0}$ of $\tilde{K}$
by the relations
$\tilde{\phi} = \Uminus \overline{w}e^{i\bar{\lambda}_0 t}\tilde{g}$ and 
$\tilde{g} = -ie^{-i\bar{\lambda}_0 t}\overline{zw}\tilde{\phi}$.
The argument which places $\tilde{\phi}$ in $\tilde{N}_{\lambda_0}$ and
establishes the existence of $\tilde{\Phi}$ is the 
same as the one for $\phi$ above.  We can now express the image of 
$T(\lambda_0)$ as
\begin{equation}
{\rm image}\,T(\lambda_0) = \{g \in L^2(\T\times\R_n):
 \la g,\,e^{-i\bar{\lambda}_0 t}\overline{zw}\tilde{\phi}\ra = 0,\ 
 \tilde{\phi} \in \tilde{N}_{\lambda_0}\}.
\end{equation}
and the cokernel of $T(\lambda_0)$ as the subspace 
$e^{-i\bar{\lambda}_0 t}\overline{zw} \tilde{N}_{\lambda_0}$.  
Our next goal is to find an inverse image for each 
$h_1 \in \coker T(\lambda_0)$ with respect to the map $T(\lambda)$,
$\lambda \not= \lambda_0$.

At first, let $g$ and $h$ be any two functions in $L^2(\T\times\R^n)$.
By Proposition~\ref{prop:Uplus}, the scalar restriction of $T(\lambda)$
described by
\begin{equation*}
a_{g,h}(\lambda) = \la T(\lambda)g, h\ra
\end{equation*}
is a holomorphic function in the lower halfplane, with derivative
\begin{equation} \label{eq:a'}
|a'_{g,h}(\lambda)| = 
|\la w e^{-i\lambda t}(\Uplus)^2 e^{i\lambda t} wzg, h\ra|
 \les |\Im[\lambda]|^{-1}\norm[g][]\norm[h][].
\end{equation}

Now fix a particular $h_1 = e^{-i\bar{\lambda}_0 t}\overline{zw}\tilde{\phi}_1$
with $\tilde{\phi}_1 \in \tilde{N}_{\lambda_0}$ of approximately unit norm,
and suppose that $g = e^{-i \lambda_0 t}w\phi$, $\phi \in N_{\lambda_0}$.  
By construction $a_{g,h_1}(\lambda_0) = 0$ and
\begin{align*}
a'_{g,h_1}(\lambda_0) = 
-i \la \Uplus V \phi, \Uminus \overline{V} \tilde{\phi}_1\ra &= 
i\la \phi, \tilde{\phi}_1\ra \\
&= i\int_0^{2\pi}
    \la \phi(t, \cdot\,), \tilde{\phi}_1(t, \cdot\,)\ra_{L^2_x}\, dt \\
&= 2\pi i \la \Phi, \tilde{\Phi}_1\ra_{L^2_x}
\end{align*}
The last line in this chain of equations is a non-selfadjoint version
of the unitarity of propagation.  The key property is that $\tilde{\phi}_1$
solves a Schr\"odinger equation with the potential $\overline{V}$.

If the orthogonal projection of $X_{\lambda_0}$ onto $\tilde{X}_{\lambda_0}$ is
bijective, then there exists a unique unit vector $\Phi_1 \in X_{\lambda_0}$
such that
\begin{equation*}
|\la \Phi_1, \tilde{\Phi}_1\ra|\gtr 1
\end{equation*}
while $\la \Phi_1, \tilde{\Phi}'\ra = 0$ for all 
$\tilde{\Phi}' \in \tilde{X}_{\lambda_0}$ orthogonal to $\tilde{\Phi}_1$.

For the associated function $g_1 = e^{-i\lambda_0 t}w\phi_1$, this provides the
lower bound
\begin{equation*}
|a_{g_1,h_1}(\lambda)| \gtr |\lambda - \lambda_0|
\end{equation*}
while at the same time
\begin{equation*}
|a_{g_1,h'}(\lambda)| \les |\lambda - \lambda_0|^2
\end{equation*}
for all unit vectors $h' \in \coker T(\lambda_0)$ orthogonal to $h_1$.


Returning to the derivative estimate~\eqref{eq:a'}, we observe that
\begin{equation*}
\bignorm[T(\lambda)g_1 \big|_{{\rm image}\,T(\lambda_0)}][] \les 
  |\lambda - \lambda_0|.
\end{equation*}
Switching the roles of $g$ and $h$ gives the bound 
\begin{equation*}
|\la T(\lambda)g, h_1 + h'\ra| \les |\lambda - \lambda_0| \norm[g][]
\end{equation*}
for every $g \in L^2(\T\times\R^n)$ and any unit vector 
$h_1 + h' \in \coker T(\lambda_0)$. 

Recall that $T(\lambda_0)$ is an invertible map between its co-image
and image.  By continuity, the restrictions of $T(\lambda)$ to these spaces
are uniformly invertible within a small neighborhood of $\lambda_0$. 
Therefore, given $g_1$ as constructed above there exists a unique element
$g'(\lambda) \in {\rm coimage}\,T(\lambda_0)$ so that 
$T(\lambda)(g_1 + g'(\lambda)) \in \coker T(\lambda_0)$.  
The norm of $g'(\lambda)$ is of order $|\lambda - \lambda_0|$.

Let $g_{h_1}(\lambda) = g_1 + g'(\lambda)$.  This is a vector of approximately
unit norm that satisfies both 
\begin{equation*}
T(\lambda) g_{h_1}(\lambda) = C_{h_1}(\lambda - \lambda_0)h_1 + 
  {\mathcal O}(|\lambda - \lambda_0|^2)
\end{equation*}
and also $T(\lambda) g_{h_1}(\lambda) \in \coker T(\lambda_0)$.
Choose any basis $\{h_j\}$ for $\coker T(\lambda_0)$.  The desired inverse
image $T(\lambda)^{-1}h_1$ will be a linear combination (with bounded 
coefficients) of the functions $(\lambda - \lambda_0)^{-1}g_{h_j}(\lambda)$.

For any unit vector $h_2 \in {\rm image}\, T(\lambda_0)$, there exists a 
unique $g_{h_2}(\lambda)$ in the co-image of $T(\lambda_0)$ so that
\begin{equation*}
T(\lambda) g_{h_2}(\lambda) - h_2 = h' \in \coker T(\lambda_0).
\end{equation*}
The norms of $g_{h_2}$ and $h'$ will be of order $1$ and 
$|\lambda - \lambda_0|$, respectively.  Thus $T(\lambda)^{-1}h'$, and 
finally $T(\lambda)^{-1} h_2 = g_{h_2} + T(\lambda)^{-1}h'$ will both
be of bounded norm.
\end{proof}

The fact that $\Im[\lambda_0] < 0$ only played a role to the extent that we
relied upon the propagator estimates of Proposition~\ref{prop:Uplus}.  If
$\lambda_0 \in \R$ instead, these can be replaced with a weaker set of
bounds based on the mapping properties of $\Rminus(\lambda)$ along the
Real axis.

\begin{proposition} \label{prop:Uplus'}
For each $\lambda \in \Compl$, $\Im[\lambda] \le 0$, the operator
$e^{-i\lambda t}\Uplus e^{i\lambda t}$ is subject to the following estimates.
\begin{align}
\norm[e^{-i\lambda t}\Uplus e^{i\lambda t}g][L^{\frac{2n}{n-2}}(\R^n; L^2(\T))]
&\les \norm[g][L^{\frac{2n}{n+2}}(\R^n; L^2(\T))] \label{eq:UplusKRS} \\
\norm[\japanese[x]^{-1}e^{-i\lambda t}\Uplus
  e^{i\lambda t}g][L^2(\R^n\times\T)]
&\les \norm[\japanese[x]g][L^2(\R^n\times\T)] \tag{\ref{eq:UplusKRS}'}
\label{eq:Uplusweighted} \\
\norm[e^{-i\lambda t}\Uplus e^{i\lambda t}g][L^{\frac{2n}{n-2}}(\R^n; L^2(\T))]
&\les \norm[\japanese[x]g][L^2(\R^n\times\T)] \tag{\ref{eq:UplusKRS}''}
\label{eq:UplusRuizVega} 
\end{align}
Given two values $\lambda_1 \not= \lambda_2$, the difference can still be
expressed as
\begin{equation*} \tag{\ref{eq:Uplusdiff}}
e^{-i\lambda_1 t}\Uplus e^{i\lambda_1 t} -
 e^{-i\lambda_2 t}\Uplus e^{i\lambda_2 t} = -i(\lambda_1 - \lambda_2)\big(
 e^{-i\lambda_1 t}\Uplus e^{i\lambda_1 t}\big)\big(e^{-i\lambda_2 t}\Uplus
 e^{i\lambda_2 t}\big).
\end{equation*}
\end{proposition}
\begin{proof}
The order of variables is interchanged from Proposition~\ref{prop:Uplus}
so that we may work entirely on the Fourier side with respect to $t$.
By Minkowski's inequality for mixed norms~\cite{LieLos} and 
Plancherel's identity,
\begin{equation*}
\norm[\hat{g}][\ell^2_n L^{2n/(n+2)}_x] \le 
\norm[\hat{g}][L^{2n/(n+2)}_x \ell^2_n]  =
\norm[g][L^{\frac{2n}{n+2}}(\R^n; L^2(\T))]
\end{equation*}
Following the Fourier characterization of $\Uplus$ given
in~\eqref{eq:UplusRminus} leads to the statement of~\eqref{eq:UplusKRS},
\begin{alignat*}{2}
\norm[e^{-i\lambda t}\Uplus e^{i\lambda t}g][L^{\frac{2n}{n-2}}(\R^n; L^2(\T))]
 &=
\norm[(e^{-i\lambda t}\Uplus e^{i\lambda t}g)\hat{\,}\,][L^{2n/(n-2)}_x 
 \ell^2_n] &&\\
&\le \norm[(e^{-i\lambda t}\Uplus e^{i\lambda t}g)\hat{\,}\,][\ell^2_n 
  L^{2n/(n-2)}_x] 
&&\les \norm[\hat{g}][\ell^2_n L^{2n/(n+2)}_x] \\
&& &\le \norm[g][L^{\frac{2n}{n+2}}(\R^n; L^2(\T))]
\end{alignat*}
where the second to last inequality is the uniform 
$L^{\frac{2n}{n+2}}\to L^{\frac{2n}{n-2}}$
bound for $\Rminus(\lambda + n), n\in\Z$ proved in~\cite{KRS}.

A proof of~\eqref{eq:Uplusweighted} which captures the sharp constant is given
in~\cite{Simon}.  The basic argument is the same as the one above, however the
Hilbert space structure of $\japanese[x]L^2(\R^n)$ and the Plancherel
identity permit precise computation of the various norms.  Finally, the
statement~\eqref{eq:UplusRuizVega} is equivalent to the resolvent bound
\begin{equation} \label{eq:RminusL2Lp}
\norm[\Rminus(\tau)\psi][\frac{2n}{n-2}] \les 
\norm[\japanese[x]\psi][2]
\end{equation}
uniformly over all $\Im[\tau] \le 0$.  
It is conceivable that~\eqref{eq:RminusL2Lp}
can be derived directly from the resolvent estimates in~\cite{Simon}
and~\cite{KRS} by factorizing $\Rminus(\tau)$ through unweighted $L^2$.  
Theorem~3.1 of~\cite{RV} is another closely related statement, differing only
in the weights and regularity of the domain 
($\japanese[x]^{-\frac12-\eps}\dot{H}^{-\frac12}(\R^n)$ versus 
$\japanese[x]^{-1}L^2(\R^n)$).  We present
a complete proof as Lemma~\ref{lem:RminusL2Lp}, in the section devoted
to Fourier analysis.
\end{proof}

\begin{lemma} \label{lem:inverse'}
Let $w \in L^2(\R^n)$ and $z \in L^\infty(\T\times\R^n)$.  Suppose the operator
$T(\lambda_0)$ fails to be invertible at $\lambda_0 \in \R$ and that neither
$K$ nor $\overline{K}$ has a resonance at $\lambda_0$.
The solution spaces $N_{\lambda_0} \subset Y_{\lambda_0}$ and
$\tilde{N}_{\lambda_0} \subset Y_{\lambda_0}$
are nontrivial and finite dimensional, and their initial values form
finite dimensional subspaces  $X_{\lambda_0}, \tilde{X}_{\lambda_0}
\subset L^2(\R^n)$.

If the orthogonal projection from $X_{\lambda_0}$ onto $\tilde{X}_{\lambda_0}$
is bijective, and if the spaces $e^{-i\lambda_0 t}N_{\lambda_0}$ and
$e^{-i\lambda_0 t}\tilde{N}_{\lambda_0}$ 
are both contained in
$L^{\frac{2n}{n+2}}(\R^n ; L^2(\T)) + \japanese[x]^{-1}L^2(\R^n\times\T)$,
then $T(\lambda)$ is invertible for every other $\lambda$ in the lower
halfplane sufficiently close to $\lambda_0$, with the norm estimate
\begin{equation} \label{eq:inverse'}
\norm[T(\lambda)^{-1}
     (h_1 + h_2)][L^2(\T\times\R^n)]
\le C(w,z,\lambda_0)(|\lambda - \lambda_0|^{-1}\norm[h_1][]+ \norm[h_2][]).
\end{equation}
In this expression 
$h_1 \in e^{-i\lambda_0 t}\overline{zw}\tilde{N}_{\lambda_0}$, 
and $h_2$ belongs to the $L^2$-orthogonal complement of
$e^{-i\lambda_0 t}\overline{zw} \tilde{N}_{\lambda_0}$.
\end{lemma}

\begin{proof}
As in Lemma~\ref{lem:inverse}, one determines that each
$g \in \ker T(\lambda_0)$ is associated with an eigenfunction
$\phi \in N_{\lambda_0}$ by the relations $\phi = \Uplus e^{i\lambda_0 t}wg$
and $g = ie^{-i\lambda_0 t}zw\phi$.  Because the available
estimate~\eqref{eq:UplusKRS} for $\Uplus$ does not map into
$L^2(\R^n\times\T)$, the extra assumption that $\lambda_0$ is not a resonance
is required in order to place $\phi \in N_{\lambda_0}$.  It then follows that
$\ker T(\lambda_0) = e^{i\lambda_0 t}w N_{\lambda_0}$ and 
$\coker T(\lambda_0) = e^{-i\lambda_0 t}\overline{zw}\tilde{N}_{\lambda_0}$.  

The next step is again to evaluate $T(\lambda)^{-1}h_1$ for $h_1 \in
\coker T(\lambda_0)$ using the function $a_{g,h}(\lambda) = 
\la T(\lambda)g, h\ra$ as a guide.  While $a_{g,h}(\lambda)$ is holomorphic
inside the lower halfplane, in general one expects it to be merely continuous
at the boundary, based on Corollary~\ref{cor:compact}.

Better behavior occurs locally if $h \in \coker T(\lambda_0)$.
Choose any $h_1 = e^{-i\lambda_0 t} \overline{zw}\tilde{\phi}_1$,
$\tilde{\phi}_1 \in \tilde{N}_{\lambda_0}$.  By construction, 
$a_{g,h_1}(\lambda_0) = 0$, and the statements in Proposition~\ref{prop:Uplus'}
imply the local Lipschitz bound
\begin{align*}
|a_{g,h_1}(\lambda)| &= |\lambda - \lambda_0| |\la wzg, e^{-i\lambda t}\Uminus
 e^{i(\lambda - \lambda_0) t}\tilde{\phi}_1\ra| \\
&\les |\lambda - \lambda_0| \norm[g][L^2(\R^n\times\T)]
 \norm[e^{-\lambda_0 t}\tilde{\phi}_1][L^{\frac{2n}{n+2}}(R^n; L^2(\T))
    + \la x\ra^{-1}L^2(\R^n\times\T)]
\end{align*}
for all $\lambda$ in the lower halfplane.  A similar bound holds for
$a_{g_1, h}(\lambda)$, where $g_1 \in \ker T(\lambda_0)$ and $h$ is any
vector in $L^2(\R^n\times\T)$.  We do not claim any differentiability 
unless both $g = e^{-i\lambda_0} w\phi \in \ker T(\lambda_0)$ and
$h_1 \in \coker T(\lambda_0)$.
In that case,
\begin{align*}
a_{g, h_1}(\lambda) &= (\lambda - \lambda_0)\la e^{-i\lambda_0 t}\phi,
  e^{-i\lambda t}\Uminus e^{i\lambda t}\overline{w} h_1 \ra \\
&= i(\lambda - \lambda_0)\la \phi, \tilde{\phi}_1 \ra 
 \ + \ (\lambda - \lambda_0)^2 \la e^{-i\lambda_0 t}\phi, 
  e^{-i\lambda t}\Uminus e^{i(\lambda - \lambda_0)t} \tilde{\phi}_1 \ra \\
&= 2\pi i(\lambda - \lambda_0)
   \la \Phi, \tilde{\Phi}_1\ra_{L^2_x} \\
&\hskip .2in
   + {\mathcal O}\big(|\lambda - \lambda_0|^2\norm[e^{-i\lambda_0 t}\phi][]
    \norm[e^{-i\lambda_0 t}\tilde{\phi}_1][] \big).
\end{align*}
The norms in the last line can be taken with respect to
$L^{\frac{2n}{n+2}}(\R^n; L^2(\T)) + \japanese[x]^{-1}L^2(\R^n\times\T)$,
since $e^{-i\lambda t}\Uminus e^{i\lambda t}$ maps this space to its dual
(see Proposition~\ref{prop:Uplus'}).
Once again the finite dimensionality of $N_{\lambda_0}$ and 
$\tilde{N}_{\lambda_0}$ makes every norm space for 
$e^{-i\lambda_0 t}\phi$ equivalent to $\norm[g][L^2(\R^n\times\T)]$
and similarly for $\tilde{\phi}_1$ and $h_1$.

If the projection of $X_{\lambda_0}$ onto $\tilde{X}_{\lambda_0}$ is
bijective, then for a fixed unit vector $h_1 \in \coker T(\lambda_0)$
there exists a unique unit vector $g_1 \in \ker T(\lambda_0)$ with the
properties
\begin{align*}
|\la T(\lambda)g_1, h_1\ra| &\gtr |\lambda - \lambda_0| \\
\bignorm[T(\lambda)g_1\big|_{{\rm image}\,T(\lambda_0)}][] &\les 
|\lambda-\lambda_0| \\
|\la T(\lambda)g_1, h'\ra| &\les |\lambda - \lambda_0|^2
\end{align*}
for all $\lambda$ in a small neighborhood of $\lambda_0$ in the lower
halfplane, and all unit vectors $h'\in \coker T(\lambda_0)$ orthogonal
to $h_1$.

From this point onward one can follow the proof of Lemma~\ref{lem:inverse}
exactly.  By continuity, $T(\lambda)$ is an invertible map between the
co-image and image of $T(\lambda_0)$.  Given $g_1$ with the properties above
there exists a unique $g'(\lambda) \in {\rm coimage}\,T(\lambda_0)$ with 
$\norm[g'(\lambda)][] \les |\lambda - \lambda_0|$ so that
$T(\lambda)(g_1 + g'(\lambda)) \in \coker T(\lambda_0)$.  The combined vector
$g_{h_1}(\lambda) = g_1 + g'(\lambda)$ is still of approximately unit norm and
satisfies 
\begin{equation*}
T(\lambda) g_{h_1}(\lambda) = C_{h_1}(\lambda - \lambda_0)h_1 +
  {\mathcal O}(|\lambda - \lambda_0|^2)
\end{equation*}
with the error lying entirely in $\coker T(\lambda_0)$.  After choosing a 
(finite) basis for $\coker T(\lambda_0)$, the true inverse 
$T(\lambda)^{-1}h_1$ is a linear combination of 
$(\lambda - \lambda_0)^{-1}g_{h_j}(\lambda)$.

The inverse image of $h_2 \in {\rm image}\,T(\lambda_0)$ is first approximated
by considering the restricted operator $T(\lambda): {\rm coimage}\,T(\lambda_0)
\to {\rm image}\,T(\lambda_0)$.  This may produce an error
$h' \in \coker T(\lambda_0)$ which can be removed via a correction of size
proportional to that of $h_2$.
\end{proof}

\begin{corollary} \label{cor:inverse}
Let $V = w^2z$ be a complex potential in $L^{n/2}_xL^\infty_t$.
Suppose the associated Floquet operators $K$ and $\overline{K}$ have no
resonances on the real axis, that condition~\ref{C1} is satisfied at every
real eigenvalue, and condition~\ref{C3} at every eigenvalue.

Then $K$ has finitely many eigenvalues $\lambda_j$, counted with multiplicity,
inside the strip $\lambda \in \Omega^-$.  Similarly,
$\overline{K}$ has only the eigenvalues $\bar{\lambda}_j$ in the reflected
strip $\Omega^+ = \{\bar{\lambda}: \lambda \in \Omega^-\}$.

For all $\lambda \in \Omega^-$,
the action of $T(\lambda)^{-1}$ is governed by the bound
\begin{align} \notag
\norm[T(\lambda)^{-1}g][L^2(\T\times\R^n)] &\les
 \norm[g][L^2(\T\times\R^n)] + \sum_j |1+ i \cot \pi(\lambda - \lambda_j)|
\, \bigg| \int_0^{2\pi}
 \la g, e^{-i\bar{\lambda}t}\overline{zw}\tilde{\phi}_j
    \ra_{L^2_x}\,dt\bigg|
\\
&= \norm[g][L^2(\T\times\R^n)] + \sum_j |1+ i \cot \pi(\lambda - \lambda_j)|
 \,\Big| \Big\la g, e^{-i\bar{\lambda}t}\overline{zw}\tilde{\phi}_j 
    \big|_{t\in[0,2\pi]} \Big\ra_{L^2_tL^2_x} \Big|   \label{eq:MainInverse}
\end{align}
where $\tilde{\phi}_j \in \tilde{N}_{\lambda_j}$ enumerate the
linearly independent eigenvectors of $\overline{K}$ with eigenvalues in 
$\Omega^+$.
\end{corollary}
\begin{proof}
The continuity and norm-decay properties of Corollary~\ref{cor:compact}
imply that $T(\lambda)^{-1}$ is invertible for all $\lambda$ in an open
subset of $\Omega^-$, with uniform bounds once $\Im[\lambda]$ is sufficiently
large.  Its complement is therefore compact in $\Omega^-$. 
If conditions~\ref{C1} and ~\ref{C3} are satisfied, then 
Lemmas~\ref{lem:inverse} and~\ref{lem:inverse'} show that the complement is
discrete as well, making it a finite set.  At each point where 
$T(\lambda)^{-1}$ fails to exist, the corresponding eigenvalues of $K$ and
$\overline{K}$ have finite multiplicity as a consequence of the Fredholm
Alternative.

For the quantitative statement, first recall that $T(\lambda+1) =
e^{-i t}T(\lambda)e^{i t}$.  This makes $\norm[T(\lambda)^{-1}][]$ periodic
with respect to integer translations.  A finite number of local statements
such as~\eqref{eq:inverse} and~\eqref{eq:inverse'} is sufficient to completely
categorize the singularities of $T(\lambda)^{-1}$ in the entire lower
halfplane.

The conclusion~\eqref{eq:MainInverse} rewrites these local bounds
to make them periodic and gathers them into a finite sum.  For example 
the single pole $(\lambda - \lambda_j)^{-1}$ is replaced with a cotangent
function.  The alterations to the inner product are designed to express
projection onto the cokernel of $T(\lambda)$ as a periodic operation. 
Note that $\coker T(\lambda+1) = e^{-i t}\coker T(\lambda)$
for every $\lambda$, and $N_{\lambda+1} = N_{\lambda}$ exactly.
In the neighborhood of $\lambda_j$ we have the estimate
\begin{equation*}
\int_0^{2\pi} \la g, e^{-i(\bar{\lambda}-\bar{\lambda}_j)t} 
  h \ra_{L^2_x}\, dt 
= \la g, h\ra_{L^2(\T\times\R^n)} + {\mathcal O}(|\lambda - \lambda_j|)
  \norm[g][2]\norm[h][2]
\end{equation*}
and it is bounded everywhere by $(1 + e^{2\pi\Im[\lambda_j - \lambda]})
 \norm[g][2]\norm[h][2]$.  Choosing a specific unit vector $h_j$ gives us
\begin{align*}
|1+i\cot \pi(\lambda - \lambda_j)|
 &\bigg|\int_0^{2\pi} \la g, e^{-i(\bar{\lambda}-\bar{\lambda}_j)t}
  h_j(t,\,\cdot\,) \ra_{L^2_x}\, dt \bigg| \\
&= \sup_{m\in\Z} \frac{1}{\pi|\lambda - (\lambda_j + m)|}
   \big|\la g, e^{-imt}h_j  \ra_{L^2(\T\times\R^n)}\big|
  + {\mathcal O}(\norm[g][])
\end{align*}
in each neighborhood of $\lambda_j + \Z$ and it is bounded by $\norm[g][]$
over the remainder of $\Omega^-$.  (To construct the global bound
we have used the fact that $|1 + i\cot \pi(\lambda)| \sim
e^{2\pi\Im[\lambda]}$ as $\Im[\lambda] \to -\infty$.)  Taking 
$h_j = e^{-i\bar{\lambda}_j t}\overline{zw}\tilde{\phi}_j$, the expression
in~\eqref{eq:MainInverse} is seen to possess the same poles 
as~\eqref{eq:inverse} and~\eqref{eq:inverse'} near each point
$\lambda_j + \Z$ and the appropriate global bound away from these
singularities.
\end{proof}

\section{Proof of Theorem \ref{thm:main}}

Based on the solution formula~\eqref{eq:Duhamel}, it suffices to show that
$(I -iw\Uplus wz)^{-1}w\Uplus f \in L^2_tL^2_x$, with support on the time
halfline $t \in [0,\infty)$.  The method of choice is suggested 
by~\eqref{eq:goal}, namely to demonstrate the finiteness of
\begin{align*}
\sup_{\mu \le 0}\norm[e^{\mu t}(I - iw\Uplus wz)^{-1}w\Uplus f][L^2_tL^2_x]^2
 &= \sup_{\mu \le 0}\int_0^1 \norm[T(\lambda'+i\mu)^{-1}
  e^{-i(\lambda'+i\mu)t}P_{\lambda' + i\mu} w\Uplus f][L^2(\T\times\R^n)]^2\,
 d\lambda' \\
 &= \sup_{\mu \le 0}\int_0^1 \norm[T(\lambda'+i\mu)^{-1}e^{-i\lambda' t}
P_{\lambda'} e^{\mu t}w\Uplus f][L^2(\T\times\R^n)]^2\, d\lambda'.
\end{align*}
Using the inequality~\eqref{eq:MainInverse} to control the behavior of
$T(\lambda'+i\mu)^{-1}$, we are left to show that
\begin{align}
\int_0^1 &\norm[P_{\lambda'}e^{\mu t}w\Uplus f][2]^2\,d\lambda' \notag \\
&+ \sum_j \int_0^1 |1 + i\cot \pi(\lambda' - \lambda_j' +i(\mu - \mu_j))|^2
\,\Big| \Big\la e^{\mu t}\Uplus f, 
 P_{\lambda'}\big(e^{-\mu t}\overline{V}\tilde{\phi}_j\big|_{t\in[0,2\pi]}\big)
   \Big\ra_{L^2_tL^2_x} \Big|^2\, d\lambda' \les \norm[f][2]^2   
\label{eq:newgoal}
\end{align}
uniformly in $\mu \le 0$.  To write things in this form we have taken advantage
of the facts that $P_{\lambda'}$ is self-adjoint on $L^2_tL^2_x$ and commutes
with pointwise multiplication by $w(x)$.

The first integral above is exactly 
$\norm[e^{\mu t}w\Uplus f][L^2_tL^2_x]^2 \les \norm[f][2]$ 
as a result of the Plancherel identity~\eqref{eq:Plancherel} and the
free Strichartz inequality~\eqref{eq:KeelTao}.  The second integral appears
more complicated, but it is also evaluated (separately for each $j$) using
Plancherel's identity in the $\lambda'$ variable.
Designate by $b_{j,\mu}(\lambda')$ the function
\begin{equation} \label{eq:b}
b_{j,\mu}(\lambda') =
 \big[1+i\cot \pi(\lambda' - \lambda_j' +i(\mu - \mu_j))\big]
  \Big\la e^{\mu t}\Uplus f, P_{\lambda'}\big(e^{-\mu t}\overline{V}
  \tilde{\phi}_j\big|_{t\in[0,2\pi]}\big) \Big\ra_{L^2_x L^2_t}.
\end{equation}
The desired bound~\eqref{eq:newgoal} is achieved by showing that
\begin{equation*}
\|b_{j,\mu}\|_{L^2([0,1])} \le C_j\norm[f][2] 
\end{equation*}
for each $j$ and all $\mu \le 0$.

Let $k \in \Z$ be the Fourier variable dual to $\lambda'$.  Given any function
$g \in L^2_tL^2_x$ and a multiplier $M(\lambda')$, the inverse Fourier 
transform of $M(\lambda')P_{\lambda'}g$ has the form
\begin{equation*}
 (MP_{\lambda'})\check{\phantom{i}}g(k,t,x) = \sum_{m\in \Z}
 \check{M}(k-m)g(t+2\pi m,x).
\end{equation*}
Integration inside the infinite sum is justified in the same manner as the
Fourier inversion formula.  The fact that $P_{\lambda'}$ resides in the
conjugate-linear half of an inner product creates some minor bookkeeping
issues.  When we wish to find the inverse Fourier transform of a function
$B(\lambda') = M(\lambda')\la F, P_{\lambda'}g\ra$, the end result is instead
\begin{equation*}
\check{B}(k) = \sum_{m\in\Z} \la F, \overline{\check{M}(k+m)}g(t+2\pi m,x)\ra.
\end{equation*}
The multiplier of interest, $M(\lambda') = 1 + i\cot \pi(\lambda'-\lambda_j'
 + i(\mu - \mu_j))$, has as its inverse Fourier transform
\begin{equation} \label{eq:Mhat}
\check{M}(k) = \big(e^{2\pi i\lambda_j'}e^{2\pi(\mu-\mu_j)}\big)^k \times
\begin{cases}
 -2\big|_{k \ge 1} \ \ &{\rm if}\ 
 \mu \le \mu_j \\
  2\big|_{k \le 0} \ \  &{\rm if}\
 \mu > \mu_j
 \end{cases}
\end{equation}
We have chosen to handle the case $\mu = \mu_j$ by analytic continuation
from $\mu < \mu_j$ rather than as a principal value. 
For our purposes the distinction is
irrelevant, as the inner product in~\eqref{eq:newgoal} will be made to
vanish wherever there is a singularity of the cotangent function.

We are now prepared to evaluate $\norm[b_{j,\mu}][2]$. 
First consider the case $\mu \le \mu_j$.  Applying
the top line from~\eqref{eq:Mhat} to the function
$g(t,x) = e^{-\mu t}\overline{V}\tilde{\phi}_j \chi_{t\in [0,2\pi]}$ and
recalling the periodicity relation for $\tilde{\phi}_j$ yields
\begin{equation*}
\check{B}(k) = -2\big(e^{-2\pi i\lambda_j'}e^{2\pi(\mu- \mu_j)}\big)^k
 \big\la F, e^{-\mu t}\overline{V}\tilde{\phi}_j \big|_{t \le 2\pi k} \big\ra.
\end{equation*}
After substituting $F(t,x) = e^{\mu t}\Uplus f$ into this expression,
Plancherel's identity tells us that 
\begin{equation}
\|b_{j,\mu}\|_{L^2([0,1])}^2 =
\sum_{k \in \Z} 4e^{4\pi(\mu - \mu_j)k} \Big|\Big\la f, \Uminus \big(
   \overline{V}\tilde{\phi}_j \big|_{t \le 2\pi k}\big)(0,\,\cdot\,)
  \Big\ra_{L^2(\R^n)}\Big|^2.
\end{equation}

The support of $\Uminus(\overline{V}\tilde{\phi}_j\big|_{t \le 2\pi k})$ is
contained within the time interval $t \in (-\infty, 2\pi k]$, therefore the
inner product vanishes for each $k \le 0$ (It vanishes when $k=0$ because of
local $L^2$ continuity).  For each $k \ge 1$ we use the eigenvector property
$\tilde{\phi}_j = \Uminus \overline{V}\tilde{\phi}_j$ and the periodicity of
$e^{-i\bar{\lambda}_j t} \tilde{\phi}_j$ to assert that
\begin{align*}
\Uminus \big(\overline{V}\tilde{\phi}_j\big|_{t \le 2\pi k}\big)(0,\,\cdot\,)
 &= \Big[\Uminus(\overline{V}\tilde{\phi}_j) 
        - \Uminus\big(\overline{V}\tilde{\phi}_j\big|_{t > 2\pi k}\big)\Big]
    (0,\,\cdot\,) \\
 &= \tilde{\Phi}_j - e^{2\pi i \bar{\lambda}_j k}e^{2\pi ik\Delta}
    \tilde{\Phi}_j
\end{align*}
with the conclusion
\begin{align*}
\|b_{j,\mu}\|_{L^2([0,1])}^2 &\le 8\sum_{k \ge 1} 
  e^{4\pi (\mu - \mu_j)k} |\la f, \tilde{\Phi}_j\ra|^2
  + 8\sum_{k \ge 1} e^{4\pi \mu k}|\la f, e^{2\pi ik\Delta}\tilde{\Phi}_j\ra|^2
\\
&\les |\mu - \mu_j|^{-1} |\la f, \tilde{\Phi}_j\ra|^2
  + |\mu| ^{-1} \norm[f][2]^2 \norm[\tilde{\Phi}_j][2]^2.
\end{align*}
If $\mu_j < 0$, then we have shown that $\norm[b_{j,\mu}][] 
\les |\mu_j|^{-1/2}\norm[f][]$
for all $f \in L^2(\R^n)$ orthogonal to $\tilde{\Phi}_j$ and all
$\mu \le \mu_j$.  The extra assumption~\ref{C2} is unnecessary in this case.

The calculations are more delicate when $\mu_j = 0$ because the unitarity of
$e^{2\pi ik\Delta}$ on $L^2$ does not provide a satisfactory estimate of
the inner product.  In its place we use the bound
\begin{equation} \label{eq:discreteKato}
\sum_{k\in\Z} |\la e^{-2\pi ik\Delta}f, \psi\ra|^2 \les
 \norm[f][2]^2 \|\psi\|_{\japanese[x]^{-1}L^2 + W^{1,2n/(n+2)}}^2,
\end{equation}
which is proved as Lemma~\ref{lem:discreteKato} in the last section.  This
is essentially a discrete-time version of more familiar Kato smoothing
estimates
\begin{equation*}
\int_\R |\la e^{-it\Delta}f, \psi\ra|^2\, dt \les
 \norm[f][2]^2 \|\psi\|_{\japanese[x]^{-1}L^2 + L^{2n/(n+2)}}^2
\end{equation*}
gathered from~\cite{Simon} and~\cite{RV}.
It is worth re-iterating that $\tilde{\Phi}_j$ has approximately unit norm
in any space that contains the finite-dimensional subspace
$\tilde{X}_{\lambda_j}$.

The remaining case $\mu_j < \mu \le 0$ is treated similarly.  The same sequence
of computations using the appropriate case of~\eqref{eq:Mhat} leads to the
identity
\begin{equation*}
\|b_{j,\mu}\|_{L^2([0,1])}^2 =
\sum_{k \in \Z} 4e^{4\pi(\mu - \mu_j)k} \Big|\Big\la f, \Uminus \big(
   \overline{V}\tilde{\phi}_j \big|_{t \ge 2\pi k}\big)(0,\,\cdot\,)
  \Big\ra_{L^2(\R^n)}\Big|^2.
\end{equation*}
This time the properties of $\tilde{\phi}_j$ simplify the inner product so that
\begin{align*}
\|b_{j,\mu}\|_{L^2([0,1])}^2 &= 
 4 \sum_{k \le 0} e^{4\pi(\mu - \mu_j)k}|\la f, \tilde{\Phi}_j\ra|^2
  + 4\sum_{k\ge 1} e^{4\pi \mu k} |\la f, e^{2\pi ik\Delta}\tilde{\Phi}_j\ra|^2
\\
&\les |\mu-\mu_j|^{-1}|\la f, \tilde{\Phi}_j\ra|^2 + \norm[f][2]^2
   \|\tilde{\Phi}_j\|_{\japanese[x]^{-1}L^2 + W^{1,2n/(n+2)}}^2.
\end{align*}
This concludes the proof of Theorem~\ref{thm:main}, with the exception of
the technical lemmas whose proofs are postponed until the last section.

\section{Fourier Analysis}
In this section we prove the various technical estimates employed during the
proof of Theorem~\ref{thm:main}.  A recurring theme will be the use of
Fourier restriction theorems, with particular emphasis on whether the
restriction to a sphere varies smoothly with respect to changes in radius.

\begin{lemma} \label{lem:RminusL2Lp}
The resolvents $\Rminus(\tau)$ observe the following inequality
\begin{equation} \tag{\ref{eq:RminusL2Lp}}
\norm[\Rminus(\tau)\psi][\frac{2n}{n-2}] \les
\norm[\japanese[x]\psi][2]
\end{equation}
with a constant that is uniform over the closed halfplane $\Im[\tau] \le 0$.
\end{lemma}

\begin{proof}
Let $\hat{\psi}_r(\omega) = \hat{\psi}(r,\omega)$ indicate the restriction of
$\hat{\psi}$ to the sphere with radius $r$.  Since we have assumed that
$\japanese[x]\psi \in L^2(\R^n)$, the radial derivative
$\partial_r\hat{\psi}_r(\omega) = \nabla\hat{\psi}(x) \cdot \frac{x}{|x|}$
is square-integrable with respect to spherical coordinates.  Combined with the
convexity of norms, this means
\begin{align} \label{eq:derivative}
\int_0^\infty r^{n-1} \Big(\frac{d}{dr}\big[\norm[\hat{\psi}_r][L^2(S^{n-1})]
  \big] \Big)^2 \,dr &\le
\int_0^\infty r^{n-1} \norm[\partial_r \hat{\psi}_r(\omega)][L^2(S^{n-1})]^2 \,
  dr \\
 &\les \norm[\japanese[x]\psi][2]^2. \notag
\end{align}
The left-hand side is a weighted $L^2$ norm of the derivative of
$\norm[\hat{\psi}_r][]$.
Hardy's inequality (or the Schur test when $n \ge 4$) then gives a weighted
$L^2$ estimate for $\norm[\hat{\psi}][]$ itself,
\begin{equation*}
\int_0^\infty r^{n-3} \norm[\hat{\psi}_r][L^2(S^{n-1})]^2\, dr
  \les \norm[\japanese[x]\psi][2]^2,
\end{equation*}
which is in effect a bound on $\norm[(-\Delta)^{-\frac12}\psi][2]$.
Applying the $L^p$ fractional integration bound for $(-\Delta)^{-\frac12}$
on top of this leads to the conclusion
\begin{equation} \label{eq:HLS}
\norm[\Rminus(0)\psi][\frac{2n}{n-2}]
  = C_n\norm[\psi * |x|^{2-n}][\frac{2n}{n-2}]
  \les \norm[\japanese[x]\psi][2].
\end{equation}
Applying the Cauchy-Schwartz inequality to~\eqref{eq:derivative} gives a
pointwise bound for $\norm[\hat{\psi}_r][]$ instead.
\begin{equation} \label{eq:restriction}
\norm[\hat{\psi}_r][L^2(S^{n-1})] \les r^{1-\frac{n}2}
  \norm[\japanese[x]\psi][2].
\end{equation}

The resolvent $\Rminus(\lambda)$ multiplies Fourier transforms by
$(|\xi|^2-\lambda)^{-1}$.  If $\Re[\lambda] < |\Im[\lambda]|$ then standard
estimates show that the convolution kernel of $\Rminus(\lambda)$ is bounded
pointwise by $|x|^{2-n}$, uniformly in $\lambda$ over this range.  The
conclusion of the lemma is verified by taking absolute values and
applying~\eqref{eq:HLS}.

The case $\Re[\lambda] > |\Im[\lambda]|$ requires more care.  Let $\chi$ be a
smooth function identically equal to 1 on $[\frac12,2]$ and supported on
$[\frac14,4]$.  Decompose the resolvent into two pieces,
\begin{align*}
(R_1\psi)\hat{\,}(\xi) &=
\big(1-\chi\big(|\Re[\lambda]|^{-\frac12}|\xi|\big)\big)
 (|\xi|^2-\lambda)^{-1} \hat{\psi}(\xi) \\
(R_2\psi)\hat{\,}(\xi) &=
  \chi\big(|\Re[\lambda]^{-\frac12}|\xi|\big)(|\xi|^2-\lambda)^{-1}
  \hat{\psi}(\xi)
\end{align*}
The convolution kernel associated to $R_1$ is again controlled pointwise by
$|x|^{2-n}$, making it subject to the same bound as in~\eqref{eq:HLS}.

Each restriction of $\hat{\phi}$ to the sphere radius $r$ makes the
contribution
\begin{equation*}
(\hat{\psi}_r)\check{\phantom{i}}(x) = (2\pi)^{-n} r^{n-1}
   \int_{S^{n-1}}e^{i rx\cdot \omega} \hat{\psi}_r(\omega)\, d\omega
\end{equation*}
toward the original function $\psi$.  Once the normalization is taken into
account, the Stein-Tomas theorem~\cite{Tomas} indicates that
\begin{equation} \label{eq:SteinTomas}
\norm[(\hat{\psi}_r)\check{\phantom{i}}][\frac{2n}{n-2}]
  \les r^{\frac{n}{2}} \norm[\hat{\psi}_r][L^2(S^{n-1})]
\end{equation}

Set $r_0 = |\Re[\lambda]|^{\frac12}$ and write out
$\hat{\psi}_r = (\hat{\psi}_r - \hat{\psi}_{r_0}) + \hat{\psi}_{r_0}$.  This
splits $R_2\psi$ into the sum of two pieces.
\begin{align*}
R_2\psi(x) = &\int_{\frac{r_0}{4}}^{4r_0} \chi\big({\txt \frac{r}{r_0}}\big)
    (r^2 - \lambda)^{-1} \big(\hat{\psi}_r - \hat{\psi}_{r_0}\big)
 \check{\phantom{l}}(x)\, dr \\
  &+ \int_{\frac{r_0}{4}}^{4r_0}  \big({\txt \frac{r}{r_0}}\big)^{n-1}
    \chi\big({\txt \frac{r}{r_0}}\big) (r^2 - \lambda)^{-1}
    (\hat{\psi}_{r_0})\check{\phantom{i}}\big({\txt \frac{r_0}{r}}x\big)\,dr \\
  = &I_1 + I_2.
\end{align*}

For the first integral, \eqref{eq:derivative} shows that
$r^{(n-1)/2}\hat{\psi}_r$, viewed as a $L^2(S^{n-1})$-valued function of $r$,
has a square-integrable weak derivative.  Therefore $\hat{\psi}_r$ is
H\"older-continuous of order $1/2$ in the interval
$r \in [\frac{r_0}4, 4r_0]$, with constant
$r_0^{(1-n)/2}\norm[\japanese[x]\psi][2]$.
Combined with~\eqref{eq:SteinTomas} this shows
\begin{align*}
\norm[I_1][\frac{2n}{n-2}] &\les \Big(\int_{\frac{r_0}{4}}^{4r_0}
  r^{\frac{n}{2}}r_0^{\frac{1-n}2} \frac{|r - r_0|^{1/2}}{|r^2 - \lambda|}
  \, dr\Big) \norm[\japanese[x]\psi][2] \\
 & \les r_0^{1/2} \Big(\int_{\frac{r_0}{4}}^{4r_0}
   \frac{|r - r_0|^{1/2}}{|r^2 - r_0^2|}\, dr\Big)
   \norm[\japanese[x]\psi][2]
 \ \les \ \norm[\japanese[x]\psi][2].
\end{align*}

The primary estimate for $I_2$ is that
\begin{equation*}
\norm[(\hat{\psi}_{r_0})\check{\phantom{i}}][\frac{2n}{n-2}]
 \les r_0\norm[\japanese[x]\psi][2]
\end{equation*}
by virtue of~\eqref{eq:restriction} and~\eqref{eq:SteinTomas}.  After a
suitable change of variables, this function can be transformed into $I_2$
via a singular integral operator that preserves $L^p$ norms.  The proposition
below completes the proof.
\end{proof}

\begin{proposition}
Given the cutoff $\chi$ as defined above and any $\lambda = r_0^2 + i\mu$ with
$|\mu| \le r_0^2$, the operator
\begin{equation*}
Sg = \int_{\frac{r_0}{4}}^{4r_0}  \big({\txt \frac{r}{r_0}}\big)^{n-1}
    \chi\big({\txt \frac{r}{r_0}}\big) (r^2 - \lambda)^{-1}
    g\big({\txt \frac{r_0}{r}}x\big)\,dr 
\end{equation*}
satisfies the bounds $\norm[Sg][p] \le C_p r_0^{-1}\norm[g][p]$
for every $1 < p < \infty$.
\end{proposition}
\begin{proof}
Consider the logarithmic spherical coordinates 
$(s,\omega) \in \R\times S^{n-1}$ defined by
$s = \log |x|$ and $\omega = \frac{x}{|x|}$.  The Jacobian
factor transforms the $L^p$ norms according to the rule
\begin{equation*}
\norm[g][p]^p = \int_{S^{n-1}}\int_\R |g(s,\omega)|^p e^{ns}\,ds d\omega  = 
\int_{S^{n-1}} \norm[g(\,\cdot\,, \omega)][L^p(e^{ns}\,ds)]^p \,d\omega.
\end{equation*}
In these coordinates the action of $S$ takes place entirely along the $s$
variable.  Let $\rho = \log(\frac{r}{r_0})$.  Then
\begin{align*}
Sg(s,\omega) &= r_0^{-1}\int_{-\log 4}^{\log 4} e^{n\rho} \chi(e^{\rho}) 
 \frac{g(s - \rho,\omega)}{e^{2\rho} - (1 +i\mu/r_0^2)}\,d\rho \\
 &= r_0^{-1}\, g *
 \left[\frac{e^{n\rho}\chi(e^\rho)}{e^{2\rho} - (1 + \mu/r_0^2)}\right]
  (s,\omega)
\end{align*}
where the convolution takes place in the $s$ variable only.  This is a
Calder\'on-Zygmund singular integral which can be controlled by the Hilbert
transform independently of the value of $\mu$.  The unweighted bounds for
the Hilbert transform apply here (despite the fact that $e^{ns}$ belongs to
no $A_p$ class) because the convolution kernel is supported in $[-2, 2]$
and the exponential function is essentially constant over any interval
of similar length.
\end{proof}

\begin{lemma} \label{lem:discreteKato}
The propagator of the free Schr\"odinger equation obeys the sampling
estimates
\begin{equation} \label{eq:discreteKato'}
\begin{aligned}
\sum_{k\in\Z} |\la e^{-2\pi ik\Delta}f, \psi\ra|^2 &\les
 \norm[f][2]^2\, \|\psi\|_{\japanese[x]^{-1}L^2}^2 \\
\sum_{k\in\Z} |\la e^{-2\pi ik\Delta}f, \psi\ra|^2 &\les
 \norm[f][2]^2\, \|\psi\|_{L^2 \cap L^{2n/(n+2)} \cap
     \dot{W}^{\alpha,2\gamma/(\gamma+2)}}^2,
\end{aligned}
\end{equation}
provided $\gamma \in [\frac{n+1}{2}, n+1]$ and $\alpha + \frac{1}{\gamma} > 1$.
\end{lemma}
\begin{proof}
For each $k$ the inner product $\la e^{-2\pi ik\Delta}f, \psi\ra$ represents
the integral
\begin{equation*}
(2\pi)^{-\frac{n}{2}} \int_R^n \hat{f}(\xi) \overline{\hat{\psi}(\xi)}
   e^{2\pi ik|\xi|^2}\,d\xi
\ = \ (2\pi)^{-\frac{n}{2}} \int_0^\infty s^{\frac{n-2}{2}} \int_{S^{n-1}}
 \hat{f}(s,\omega) \overline{\hat{\psi}(s,\omega)} e^{2\pi iks}\, d\omega ds
\end{equation*}
where $(s,\omega)$ are the spherical coordinates $s = |\xi|^2$, $\omega =
\frac{x}{|x|}$.  This in turn describes (up to constants)
the $k{\rm th}$ Fourier coefficient of the periodic function
$\sum_m F(s+m)$, where
\begin{equation*}
F(s) = (s)^{\frac{n-2}{2}} \int_{S^{n-1}}
  \hat{f}(s, \omega) \overline{\hat{\psi}(s, \omega)}\, d\omega.
\end{equation*}
We are therefore concerned with finding conditions on $\psi$ that lead to
$\sum_m F(s+m)$ belonging to $L^2([0,1])$.  It would be sufficient to show
instead that $F \in \ell^1_m(L^2([m, m+1]))$.

Plancherel's identity dictates that $s^{(n-2)/4}\hat{f}(s,\omega)$
is an element of $\ell^2_m L^2([m, m+1]; L^2(S^{n-1}))$.
Bounds of the type~\eqref{eq:discreteKato'} will follow provided that
$s^{(n-2)/4}\hat{\psi}$ belongs to
$\ell^2_m L^\infty([m,m+1]; L^2(S^{n-1}))$.
Taking $\hat{\psi}_s$ to be the restriction of $\hat{\psi}$ to the sphere
$|\xi|^2 = s$, we wish to show that 
\begin{equation*}
\sum_{m\ge 0} \sup_{s \in [m,m+1]} s^{\frac{n-2}{2}} 
   \norm[\hat{\psi}_s][L^2(S^{n-1})]^2
\end{equation*}
is controlled by the norm of $\psi$ in a space of our choosing.

Suppose $\psi \in \japanese[x]^{-1}L^2$.  Changing variables from $r$ to $s$
in~\eqref{eq:derivative} leads to the derivative estimates
\begin{equation} \label{eq:derivative'}
\int_0^\infty s^\frac{n}{2} 
 \Big(\frac{d}{ds}\big[\norm[\hat{\psi}_s][L^2(S^{n-1})]
  \big] \Big)^2 \,ds \les \norm[\japanese[x]\psi][2]^2.
\end{equation}
Local differences are estimated by the mean value theorem and
Cauchy-Schwartz.  For any pair of points $s_1, s_2 \in [m, m+1]$,
\begin{align*}
\Big| \norm[s_1^{(n-2)/4}\hat{\psi}_{s_2}][] - 
   \norm[s_2^{(n-2)/4}\hat{\psi}_{s_1}][]\Big|^2
&\le \int_m^{m+1} 
  \Big(\frac{d}{ds}\big[s^{(n-2)/4}\norm[\hat{\psi}_s][L^2(S^{n-1})]\big]
    \Big)^2 \,ds \\
&\le 
  2\int_m^{m+1} s^{\frac{n-6}{2}}\norm[\hat{\psi}_s][]^2
+ s^{\frac{n-2}{2}}\Big(\frac{d}{ds} \norm[\hat{\psi}_s][]
  \Big)^2\,ds
\end{align*}
The $L^\infty$ norm of a positive function over a unit interval is
controlled by its integral and the variation of its values, hence
\begin{equation}
\begin{aligned}
\sum_{m \ge 1} \sup_{s\in[m,m+1]} 
   s^{\frac{n-2}{2}}\norm[\hat{\psi}_s][]^2
 &\le \sum_{m\ge 1} \int_m^{m+1} \big(s^{\frac{n-2}{2}} + 2s^{\frac{n-6}{2}}
 \norm[\hat{\psi}_s][]^2
 + 2 s^{\frac{n-2}{2}}
  \Big(\frac{d}{ds}\norm[\hat{\psi}_s][] \Big)^2 \,ds \\
 &\les \int_1^\infty s^{\frac{n-2}{2}}\norm[\hat{\psi}_s][]^2\,ds
   +  \int_1^\infty s^{\frac{n}{2}}
  \Big(\frac{d}{ds}\big[\norm[\hat{\psi}_s][] \big] \Big)^2 \,ds \\
 &\les \norm[\psi][2]^2 + \norm[\japanese[x]\psi][2]^2
\end{aligned}
\end{equation}
by Plancherel and~\eqref{eq:derivative'}, respectively.  The supremum over the
interval $s\in [0,1]$ is controlled separately by the estimate
\begin{align*}
s^{\frac{n-2}{4}}\norm[\hat{\psi}_s][] \le s^{\frac{n-2}{4}} 
  \int_s^\infty \Big|\frac{d}{ds}\big[\norm[\hat{\psi}_s][] \big] \Big|\,ds
 &\les \bigg(\int_s^\infty s^{\frac{n}{2}}
   \Big|\frac{d}{ds}\big[\norm[\hat{\psi}_s][] \big] \Big|^2\,ds \bigg)^{1/2}\\
 &\les \norm[\japanese[x]\psi][2]
\end{align*}
which is a combination of Cauchy-Schwartz and~\eqref{eq:derivative'}.

For the second statement, the condition $\psi \in L^{2n/(n+2)}$ is
most important in the interval $s \in [0,1]$ and the Sobolev regularity
condition plays a major role as $s \to \infty$. It is clearly necessary to have
$\psi \in L^2$, otherwise the inner product in~\eqref{eq:discreteKato'}
could be undefined for one or more values of $k$.

The dual statement to~\eqref{eq:SteinTomas}, when normalized with the correct
factor of $r^{n-1}$ indicates that $s^{(n-2)/4}\norm[\hat{\psi}_s][] \les
\norm[\psi][\frac{2n}{n+2}]$ for all $s > 0$.  In particular, the
supremum over $s \in [0,1]$ is bounded in this manner.

The fact that $\psi \in L^2$ implies that $s^{(n-2)/4}\norm[\hat{\psi}_x][]^2$
is integrable.  Controlling its $L^\infty$ norm on a unit interval in terms of
its $L^1$ norm generally requires some degree of continuity.  In the previous
case we were able to infer differentiability of $\hat{\psi}_s$ from the
polynomial weighted decay of $\psi$.  With $\psi$ merely belonging to an $L^p$
space, it may still be true that $\hat{\psi}$ is continuous, but the modulus of
continuity is not determined by $\norm[\psi][]$ alone.  We exploit the
observation (also used in~\cite{GS2}), that the norm of $\hat{\psi}_s$
varies smoothly even when the restrictions themselves do not.

\begin{proposition} \label{prop:restrictions}
Let $\gamma \in [\frac{n+1}{2}, n+1]$.  The Fourier restrictions of 
$\psi \in L^{\frac{2\gamma}{\gamma+2}}(\R^n)$ satisfy the continuity bound
\begin{equation} \label{eq:restrictions}
m^{\frac{n-2}{2}}
  \big(\norm[\hat{\psi}_{s_1}][]^2 - \norm[\hat{\psi}_{s_2}][]^2\big)
 \les \Big(\frac{|s_1 - s_2|^{n+1-\gamma}}{m}\Big)^{1/\gamma} 
\norm[\psi][\frac{2\gamma}{\gamma+2}]^2
\end{equation}
for every pair $s_1, s_2 \in [m,m+1]$, $m\ge 1$.
\end{proposition}

The power of $|s_1 - s_2|$ does not matter much so long as it is nonnegative.
Of considerably greater interest is the factor of $m^{-1/\gamma}$,
as it contributes meaningfully to the bound
\begin{equation*}
\Big|\norm[s_1^{(n-2)/4} \hat{\psi}_{s_1}][]^2 - 
\norm[s_2^{(n-2)/4} \hat{\psi}_{s_2}][]^2 \Big|  
\les \big(m^{-(\alpha + \frac{1}{\gamma})} + 
  m^{-(\alpha + 2 - \frac{n}{\gamma})}\big)
   \norm[\psi][\dot{W}^{\alpha,2\gamma/(\gamma+2)}]^2
\end{equation*}
for each pair $s_1, s_2 \in [m,m+1]$, $m \ge 1$.  The first term is derived
from~\eqref{eq:restrictions}, and the second (which is dominated by the
first) from the Stein-Tomas theorem.  As before, the $L^\infty$
norm of a function on a unit interval is controlled by the its $L^1$ norm 
and the diameter of its image.  Consequently,
\begin{align*}
\sum_{m\ge 0} \sup_{s\in[m,m+1]} \norm[s^{(n-2)/4}\hat{\psi}_s][]^2
 &\les \norm[\psi][\frac{2n}{n+2}]^2 
 + \sum_{m\ge 1} \bigg(\int_m^{m+1} s^{(n-2)/2}\norm[\hat{\psi}_s][]^2\,ds 
    + m^{-(\alpha + \frac{1}{\gamma})} 
     \norm[\psi][\dot{W}^{\alpha,2\gamma/(\gamma+2)}]^2 \bigg) \\
 &\les  \norm[\psi][\frac{2n}{n+2}]^2 
   + \norm[\psi][\dot{W}^{\alpha,2\gamma/(\gamma+2)}]^2
   + \int_1^\infty s^{(n-2)/2}\norm[\hat{\psi}_s][]^2\,ds \\
 &= \norm[\psi][\frac{2n}{n+2}]^2 +
    \norm[\psi][\dot{W}^{\alpha,2\gamma/(\gamma+2)}]^2 +\norm[\psi][2]^2
\end{align*}
provided the sum of $m^{-(\alpha + 1/\gamma)}$ is convergent.
\end{proof}
\begin{proof}[Proof of Proposition \ref{prop:restrictions}]
On each interval $[m,m+1]$ the function $s^{(n-2)/4}$ can be replaced by
the constant $m^{(n-2)/4}$.  Recalling the proof of the Stein-Tomas
theorem, Fourier restriction to the sphere is described by a convolution
operator, with the $TT^*$ estimate
\begin{equation} \label{eq:TT*}
\norm[\hat{\psi}_s][]^2 = C_n \int_{\R^{2n}} f(x) K\big(s^{\frac12}(x-y)\big)
 \overline{f(y)}\,dx dy.
\end{equation}
The kernel is an oscillatory function bounded pointwise by 
$|K(z)| \les \japanese[z]^{-(n-1)/2}$.  The related function
$\tilde{K}(z) = zK'(z)$ is also oscillatory, and bounded pointwise
by $\japanese[z]^{-(n-3)/2}$.  If $f$ is a Schwartz function it is permissible
to differentiate~\eqref{eq:TT*} with respect to $s$, obtaining
\begin{equation*}
\frac{d}{ds}\Big(\norm[\hat{\psi}_s][]^2\Big) = C_n s^{-1}\int_{\R^{2n}}
 f(x) \tilde{K}\big(s^{\frac12}(x-y)\big)\overline{f(y)}\, dx dy.
\end{equation*}
The same interpolation argument that proves the Stein-Tomas theorem also
suffices to show that convolution with $\tilde{K}$ is a bounded operator
from $L^{(2n+2)/(n+5)}(\R^n)$ to its dual space $L^{(2n+2)/(n-3)}(\R^n)$.
Combining this with the usual restriction estimate and scaling appropriately,
\begin{equation*}
m^{(n-2)/2}\Big|\norm[\hat{\psi}_{s_1}][]^2 - \norm[\hat{\psi}_{s_2}][]^2\Big|
 \les \max\Big( m^{1/(n+1)}\norm[f][\frac{2n+2}{n+3}]^2,\  
            m^{2/(n+1)}|s_1 - s_2|\norm[f][\frac{2n+2}{n+5}]^2\Big).
\end{equation*}
These represent the cases $\gamma = n+1$ and $\gamma = \frac{n+1}{2}$,
respectively.  The intermediate cases follow from Riesz-Thorin interpolation,
noting that the norm of a self-adjoint linear operator agrees with the
extremal value of its quadratic form.
\end{proof}

\begin{remark}
The proof of Lemma~\ref{lem:discreteKato} hinges on placing the spherical
restrictions of $\hat{\psi}$ inside a mixed-norm space $\ell^2(L^\infty)$ with
respect to the radial variable.  This consists of three essentially independent
requirements.
\renewcommand{\theenumi}{\arabic{enumi}}
\renewcommand{\labelenumi}{\theenumi.}
\begin{enumerate}
\item Because of the embedding $\ell^2(L^\infty) \subset \ell^2(L^2)$ and
the Plancherel identity, we must have $\psi \in L^2$.  This is the only way
to produce $\ell^2$ decay as $m \to \infty$.
\item Since $\ell^2(L^\infty)$ also embeds into $\ell^\infty(L^\infty)$, 
 the normalized restrictions $s^{(n-2)/4} \hat{\psi}_s$ must be uniformly
 bounded, in particular as $s \to 0$.  This is achieved so long as $\psi$
 belongs to either of the spaces $\japanese[x]^{-1}L^2$ or $L^{2n/(n+2)}$.
\item The norm of the restrictions must also be sufficiently continuous
 so that the $\ell^2(L^2)$ bound implied by the first item can be improved
 into $\ell^2(L^\infty)$.

Proposition~\ref{prop:restrictions} provides one estimate for the modulus
of continuity of $\norm[\hat{\psi}_s][]^2$ based on the Stein-Tomas
restriction theorem.  Another estimate, based on $L^2$ trace
properties, is available when $\japanese[x]^{\beta}\psi \in L^2$ for some
$\beta > \frac12$.  The latter bounds are well known from the proof of the
limiting absorption principle~\cite{RS4} and spectral theory of Schr\"odinger
operators.
\end{enumerate}
The norm spaces in the statement of Proposition~\ref{prop:restrictions}
were chosen to meet these requirements entirely with weights, or entirely with
homogeneous $L^p$ conditions, respectively.  To create a more comprehensive
list, one can mix and match the two approaches in any combination. 
A precise but unwieldy formulation is presented below.
\end{remark}
\begin{proposition}
The propagator of the free Schr\"odinger equation obeys the sampling
estimates
\begin{equation*}
\sum_{k\in\Z} |\la e^{-2\pi ik\Delta}f, \psi\ra|^2 \les
 \norm[f][2] \norm[\psi][]
\end{equation*}
where the norm of $\psi$ is taken in the interpolation space
\begin{equation*}
\psi \in L^2 \cap \big(\japanese[x]^{-1}L^2 + L^{\frac{2n}{n+2}}\big)
\cap \big(\japanese[x]^{-\frac12-\eps}L^2 + \dot{W}^{\frac{n}{n+1}+\eps, 
 \frac{2n+2}{n+3}}
  + \dot{W}^{\frac{n-1}{n+1}+\eps, \frac{2n+2}{n+5}}\big).
\end{equation*}
\end{proposition}

\bibliographystyle{amsplain}

\end{document}